\documentclass[11pt]{amsart}

\usepackage{amsthm, amsfonts, amssymb, amscd}
\usepackage[pagebackref,colorlinks]{hyperref}
\usepackage{tikz-cd}
\usepackage{geometry}
\usepackage{marginnote}

\theoremstyle{definition}
\newtheorem{ntn}{Notation}[section]
\newtheorem{dfn}[ntn]{Definition}
\theoremstyle{plain}
\newtheorem{lem}[ntn]{Lemma}
\newtheorem{prp}[ntn]{Proposition}
\newtheorem{thm}[ntn]{Theorem}
\newtheorem{cor}[ntn]{Corollary}

\theoremstyle{definition}

\newtheorem{rem}[ntn]{Remark}
\newtheorem{exa}[ntn]{Example}

\numberwithin{equation}{section}

\newcommand{\N}{\mathbb{N}}
\newcommand{\z}{\mathbb{Z}}
\newcommand{\q}{\mathbb{Q}}

\newcommand{\C}{\mathbb{C}}

\newcommand{\F}{\mathbb{F}}
\newcommand{\Eb}{\mathbb{E}}

\newcommand{\EE}{\mathcal{E}}
\newcommand{\Aa}{\mathcal{A}}
\newcommand{\BB}{\mathcal{B}}
\newcommand{\GG}{\mathcal{G}}
\newcommand{\WW}{\mathcal{W}}

\newcommand{\RP}{\mathcal{RP}}

\newcommand{\RB}{\mathcal{RB}}
\newcommand{\OO}{\mathcal{O}}

\newcommand{\LL}{\mathcal{L}}

\renewcommand{\aa}{{A^\times}}
\newcommand{\tors}{{{\rm Tor}_1^{\z}}}

\newcommand{\half}{{\Big[\frac{1}{2}\Big]}}
\newcommand{\pth}{{\Big[\frac{1}{p}\Big]}}
\newcommand{\third}{{\Big[\frac{1}{3}\Big]}}

\newcommand{\mt}{\mapsto}
\newcommand{\lan}{\langle}
\newcommand{\ran}{\rangle}
\newcommand{\se}{\subseteq}
\newcommand{\arr}{\rightarrow}
\newcommand{\larr}{\longrightarrow}
\newcommand{\harr}{\hookrightarrow}
\newcommand{\two}{\twoheadrightarrow}
\newcommand{\Lan}{\langle\! \langle}
\newcommand{\Ran}{\rangle \!\rangle}

\newcommand{\stabe}{{\rm Stab}}
\newcommand{\GL}{{\rm GL}}
\newcommand{\PGL}{{\rm PGL}}
\newcommand{\PB}{{\rm PB}}
\newcommand{\PT}{{\rm PT}}
\newcommand{\PN}{{\rm PN}}

\newcommand{\SL}{{\rm SL}}
\newcommand{\PSL}{{\rm PSL}}

\newcommand{\SM}{{\rm SM}}
\newcommand{\PSM}{{\rm PSM}}
\newcommand{\GW}{{\rm GW}}

\newcommand{\Ind}{{\rm Ind}}

\newcommand{\Bb}{{\rm B}}
\newcommand{\Tt}{{\rm T}}
\renewcommand{\char}{{\rm char}}

\newcommand{\im}{{\rm im}}
\newcommand{\ind}{{\rm ind}}

\newcommand{\inc}{{\rm inc}}
\newcommand{\id}{{\rm id}}
\newcommand{\GE}{{\rm GE}}

\newcommand {\mtx}[2]
{\left(\!\!\!
\begin{array}{cc}
#1   \\
#2 
\end{array}
\!\!\!\right)}

\newcommand {\mtxx}[4]
{\left(\!\!
\begin{array}{cc}
\!\!#1 & \!\!#2   \\
\!\!#3 & \!\!#4
\end{array}\!\!
\right)
}

\newtheoremstyle{athm}
{}
{}
{\itshape}
{}
{\scshape}
{}
{.5em}
{\thmnote{#3}}
\theoremstyle{athm}
\newtheorem*{athm}{}

\begin{document}
\title{The third homology of projective special linear group of degree two}
\author{Behrooz Mirzaii}
\author{Elvis Torres P\'erez}

\address{\sf 
Instituto de Ci\^encias Matem\'aticas e de Computa\c{c}\~ao (ICMC), Universidade de S\~ao Paulo, 
S\~ao Carlos, Brazil}
\email{bmirzaii@icmc.usp.br}
\address
{\sf Faculty of Sciences, National University of Engineering (UNI), Lima, Peru}
\email{elvis.torres.p@uni.pe}

\begin{abstract}
In this paper we investigate the third homology of the projective special linear group 
$\PSL_2(A)$. As a result of our investigation we prove a projective refined Bloch-Wigner 
exact sequence over certain class of rings. The projective Bloch-Wigner exact sequence 
over an algebraically closed field of characteristic zero is a classical result and has 
many application in algebra, number theory and geometry.
\end{abstract}
\maketitle

The third homology of $\SL_2(A)$ (and $\PSL_2(A)$) appears in many areas of algebra, number theory 
and geometry. It appears in the study of the third $K$-group of the ring $A$ \cite{suslin1991}, 
\cite{hmm2022}, it has deep connection with the polylogarithm function \cite{bloch2000} and is 
fundamental part of the scissors congruence problem in 3-dimensional hyperbolic and spherical 
geometry \cite{dupont-sah1982}, \cite{sah1989} (when $A\se \C$).

In this article we investigate the third homology of $\PSL_2(A)$ and its connection with the refined 
scissors congruence group of $A$ introduced and studied by Hutchinson \cite{hutchinson-2013},  
\cite{hutchinson2013}. 

Let $A$ be a commutative local ring. Let $\GG_A$ be the square class group of $A$, i.e. $\GG_A:=\aa/(\aa)^2$.
We denote by $\lan a\ran$ the element of $\GG_A$ represented by $a\in \aa$.
Let $\RP(A)$ be the quotient of the free $\z[\GG_A]$-module generated by symbols $[a]$, 
$a\in \WW_A:=\{a\in \aa:a-1\in \aa\}$, by the $\z[\GG_A]$-submodule generated by the elements
\[
[a]-[b]+\lan a\ran\bigg[\frac{b}{a}\bigg]-\lan a^{-1}-1\ran\Bigg[\frac{1-a^{-1}}{1-b^{-1}}\Bigg] +
\lan 1-a\ran\Bigg[\frac{1-a}{1-b}\Bigg],
\]
where $a,b,a/b \in \WW_A$. The map $\lambda_1:\RP(A)\arr \z[\GG_A]$ given by $[a]\mapsto\Lan a \Ran\Lan 1-a\Ran$,
is a well-defined $\z[\GG_A]$-homomorphism, where for $a\in \aa$ we set $\Lan a\Ran:=\lan a\ran-1$. If we consider
\[
S_\z^2(\aa):=(\aa \otimes_\z \aa)/\lan a\otimes b+b\otimes a:a,b \in \aa\ran
\]
as a trivial $\z[\GG_A]$-module, then $\lambda_2: \RP(A) \arr S_\z^2(\aa)$, given 
by $\overline{[a]} \mapsto \overline{a \otimes (1-a)}$, is a homomorphism of $\z[\GG_A]$-modules. 

Hutchinson defined the refined scissors congruence group $\RP_1(A)$ and the refined Bloch group
$\RB(A)$ of $A$ as follows:
\[
\RP_1(A):=\ker\Big(\lambda_1:\RP(A) \arr \z[\GG_A]\Big), \ \ \ 
\RB(A):=\ker\Big(\lambda_2|_{\RP_1(A)}:\RP_1(A) \arr S_\z^2(\aa)\Big)
\] 
(see \cite[page 28]{hutchinson-2013}, \cite[Subsection 2.3]{hutchinson2013}). The refined Bloch group 
$\RB(A)$ of $A$ is closely related to the classical Bloch group of $A$ \cite{hutchinson-2013}, 
\cite{C-H2022}, \cite{suslin1991}.

Hutchinson formulated the first version of a refined Bloch-Wigner sequence in 
\cite[Theorem 4.3]{hutchinson-2013}. In fact, he proved that for any infinite field $F$ 
there is a complex of $\z[F^\times/(F^\times)^2]$-modules
\[
0 \arr\tors(\mu(F),\mu(F)) \arr H_3(\SL_2(F),\z)\arr\RB(F)\arr 0,
\]
which is exact at every term except possibly at the term $H_3(\SL_2(F), \z)$, where the homology of the 
complex is annihilated by $4$. Here $\mu(F)$ is the group of roots of unity in $F$.
In a later work \cite[Theorem 3.22]{hutchinson2017} he generalized this 
result to any local domain with ``sufficiently large" residue field. In \cite[page 3]{C-H2022}, Coronado and 
Hutchinson asked if the sequence
\begin{equation*}\label{problem}
0 \arr\tors(\mu(A),\mu(A))^\sim \arr H_3(\SL_2(A),\z)\arr\RB(A)\arr 0
\end{equation*}
is exact for any local domain $A$, where $\tors(\mu(A),\mu(A))^\sim$ is the unique non-trivial extension of 
$\tors(\mu(A),\mu(A))$ by $\mu_2(A)$, the group of $2$-roots of unity in $A$.

A refined Bloch-Wigner exact sequence will reduce the study of the group
$H_3(\SL_2(A),\z)$, modulo a certain known torsion subgroup, to the study of the refined 
Bloch group $\RB(A)$ and the scissors congruence group $\RP_1(A)$. The algebra of $\RP_1(A)$ has been 
studied by Hutchinson in a series of papers \cite{hutchinson2013}, \cite{hutchinson2017}, 
\cite{hutchinson-2017}, \cite{C-H2022}.

In this article we prove a refined Bloch-Wigner exact sequence involving the third homology
of the projective special linear group $\PSL_2(A)$ (see Theorem \ref{Proj-BW}, Theorem \ref{G<4}
and Corollary \ref{PRBW-cor}).

\begin{athm}[{\bf Theorem A.}]
Let $A$ be a local domain such that its residue field either is infinite or it has $p^d$ elements, 
where $(p-1)d>6$. Moreover, let either $-1$ be square or the square class group of $A$ has at most 
$4$ elements. Then we have the exact sequence of $\z[\aa/(\aa)^2]$-modules
\begin{equation*}\label{PRBW}
0 \arr\tors(\widetilde{\mu}(A),\widetilde{\mu}(A)) \arr H_3(\PSL_2(A),\z)\arr\RB(A)\arr 0,
\end{equation*}
where $\widetilde{\mu}(A):=\mu(A)/\mu_2(A)$.
\end{athm}

In particular, we show that the above result is valid for real closed fields, most of local fields 
and most of finite fields.

If we remove the restricted condition on $-1$ or the square class group of $A$, we prove the following result
(see Theorems \ref{general-exa} and \ref{general-exa-1} and Corollary \ref{cor-last}). 
Here $\PSM_2(A)$ is the group of monomial matrices in $\PSL_2(A)$.

\begin{athm}[{\bf Theorem B.}]
Let $A$ be a local domain such that its residue field either is infinite or it has $p^d$ elements, 
where $(p-1)d>6$. Then we have the exact sequence of $\z[\aa/(\aa)^2]$-modules
\begin{equation*}\label{PSM2-1}
H_3(\PSM_2(A),\mathbb{Z})\arr H_3(\PSL_2(A),\mathbb{Z}) \arr \RB(A)/\lan\psi_1(-1)\ran\arr 0,
\end{equation*}
where $\lan \psi(-1)\ran$ is the $\z[\aa/(\aa)^2]$-submodule of $\RB(A)$ generated by  $\psi(-1)$. Moreover,
\begin{equation*}\label{PSM2-2}
H_3(\PSL_2(A),\PSM_2(A);\z)\simeq \RP_1(A)/\lan\psi_1(-1) \ran.
\end{equation*}
\end{athm}

Although understanding the structure of $H_3(\PSL_2(A),\z)$ is interesting and important in geometry and 
number theory, one of our motivations for the study of this group and exact sequences such as in {\bf Theorem A} 
and {\bf Theorem B}, is the natural map from $H_3(\SL_2(A),\z)$ to the indecomposable part of the 
third $K$-group of $A$, i.e. $K_3^\ind(A)$. In fact, if $A$ is a local ring, then there is a natural map 
\[
H_3(\SL_2(A),\z) \arr K_3^\ind(A)
\]
(see \cite[\S5]{hutchinson-tao2009}). Suslin asked the following interesting question 
\cite[Question 4.4]{sah1989}.

\begin{athm}[{\bf Question (Suslin).}]
If $F$ is an infinite field, is the natural map $\alpha_F: H_3(\SL_2(F),\z)_{F^\times} \arr K_3^\ind(F)$ 
an isomorphism?
\end{athm}

It is known that if $A$ is a local ring with ``sufficiently large'' residue field, then the kernel 
and cokernel of $\alpha_A$ are $2$-groups \cite[Proposition~6.4]{mirzaii-2008}. Hutchinson and Tao 
proved that $\alpha_F$, $F$ an infinite field, always is surjective \cite[Lemma~5.1]{hutchinson-tao2009}. 
When $F$ is finite, $\alpha_F$ is studied in \cite[Proposition~6.4]{mirzaii2017}.

Here we outline the organization of the present paper. In Section~\ref{sec1} we introduce the refined
scissors congruence group, its refined Bloch group and some of their elementary properties. 
In Section~\ref{sec2} we introduce and study a spectral sequence which will be our main tool in 
handling the homology of $\PSL_2(A)$. Here we calculate some of the differentials of the 
spectral sequence which are very important for us. In particular we prove Lemmas \ref{surj1} 
and \ref{wedge3} which are fundamental for our proof of the projective Bloch-Wigner exact sequence. In 
Section~\ref{sec3} we prove our projective refined Bloch-Wigner exact sequence over some classes of rings, i.e. 
Theorem A (see Theorem \ref{Proj-BW} and Theorem \ref{G<4}). In Section~\ref{sec4}, 
we prove Theorem B (see Theorem \ref{general-exa} and Theorem \ref{general-exa-1}).\\
 ~\\
{\bf Notations.} 
In this paper all rings are commutative, except possibly group rings, and have the unit 
element $1$. For a ring $A$ let $\GG_A:=\aa/(\aa)^2$ by $\GG_A$. The element of $\GG_A$ 
represented by $a\in \aa$ will be denoted by $\lan a \ran$. Moreover, let
\[
\mu_2(A):=\{a\in A: a^2=1\}, \ \ \ \mu(A):=\{a\in A: \text{there is $n\in \N$ such that $a^n=1$} \}.
\]
If $G$ is a subgroup of $\aa$ containing $\mu_2(A)$, we denote $G/\mu_2(A)$ by $\widetilde{G}$.
We set 
\[
\PGL_2(A):=\GL_2(A)/\{aI_2: a\in \aa\}, \ \ \ \ \PSL_2(A):=\SL_2(A)/\{aI_2: a\in\mu_2(A)\}. 
\]
Note that the determinant gives us the extension $1 \arr \PSL_2(A) \arr \PGL_2(A) \arr \GG_A \arr 1$.\\
~\\
{\bf Acknowledgments.}
We would like to thank the anonymous referee for many valuable comments and suggestions which improved 
the readability and the presentation of the present article. 

\section{The refined scissors congruence group}\label{sec1}

Let $A$ be a commutative ring. A column vector 
${\pmb u}={\mtx {u_1} {u_2}}\in A^2$ is called unimodular if 
\[
u_1A+u_2A=A.
\]

For any non-negative integer $n$, let $X_n(A^2)$ be the free abelian group generated 
by the set of all $(n+1)$-tuples $(\lan{\pmb v_0}\ran, \dots, \lan{\pmb v_n}\ran)$, where 
every ${\pmb v_i} \in A^2$ is unimodular and any two vectors ${\pmb v_i}, {\pmb v_j}$, 
$i\neq j$, are a basis of $A^2$. Observe that $\lan{\pmb v}\ran={\pmb v}A$.

We consider $X_n(A^2)$ as a left $\PGL_2(A)$-module (respectively left 
$\PSL_2(A)$-module) in a natural way. If necessary, we convert this action 
to a right action by the definition $m.g:=g^{-1}m$. We define the $n$-th 
differential operator
\[
\partial_n : X_n(A^2) \arr X_{n-1}(A^2), \ \ n\ge 1,
\]
as an alternating sum of face operators which throws away the $i$-th component of generators. 
It is straightforward to check that $\partial_{n-1}\circ\partial_n=0$. The complex 
\[
X_\bullet(A^2): \cdots \larr X_2(A^2) \overset{\partial_2}{\larr} X_1(A^2) \overset{\partial_1}{\larr} 
X_0(A^2) \arr 0
\]
will be our main tool for the study of the third homology of $\PSL_2(A)$. Let $\epsilon: X_0(A^2) \arr \z$ 
be defined by $\sum_i n_i(\lan{\pmb  v_{0,i}}\ran) \mt \sum_i n_i$. 

\begin{dfn}
A ring $A$ is called a ${\pmb \GE_2}${\bf -ring} if $\SL_2(A)$ is generated by elementary 
matrices $E_{12}(a)={\mtxx 1 a 0 1}$ and $E_{21}(a)={\mtxx 1 0 a 1}$, $a\in A$.
\end{dfn}

For a definition of the unstable ${\rm K}_2$-group of degree $2$, denoted by ${\rm K}_2(2,A)$ and its 
central subgroup ${\rm C}(2,A)$ generated by Steinberg symbols we refer the reader to 
\cite[App. A]{hutchinson2022}.

\begin{prp}[Hutchinson \cite{hutchinson2022}]\label{GE2C}
{\rm (i)} $H_0(X_\bullet(A^2))\overset{\bar{\epsilon}}{\simeq} \z$ if and only if $A$ is a $\GE_2$-ring. 
\par {\rm (ii)} $H_1(X_\bullet(A^2))=0$ if and only if ${\rm K}_2(2,A)/{\rm C}(2, A)$ is a perfect group.
\end{prp} 
\begin{proof}
See \cite[Theorem 3.3, Theorem 7.2 and Corollary 7.3]{hutchinson2022}. 
\end{proof} 

\begin{exa}\label{exact-11}
We mostly will work with a ring $A$ which satisfies the condition that $X_\bullet(A^2)\arr \z$ 
is exact in dimension $<2$, i.e. $H_0(X_\bullet(A^2))\overset{\bar{\epsilon}}{\simeq} \z$ and 
$H_1(X_\bullet(A^2))=0$. The following rings satisfy this condition:
\par (1) Any local ring (see \cite[Theorem~4.1]{cohn1966}, 
\cite[Theorem 7.2 and Corollary 7.3]{hutchinson2022}),
\par (2) Any semi-local ring such that none of the rings 
$\z/2 \times \z/2$, $\z/6$ is a direct factor of $A/J(A)$, where $J(A)$ is the 
Jacobson radical of $A$ \cite[Theorem 2.14]{menal1979}.
\par (3) Any ring with many units \cite[\S2]{mirzaii-2008}.
\par (4) $A=\z[\frac{1}{m}]$, where $m$ can be expressed as a product of primes 
$m=p_1^{\alpha_1}\cdots p_t^{\alpha_t}$ ($\alpha_i\geq 1$) with property that $(\z/p_i)^\times$ 
is generated by the residue classes $\{-1, p_1,\dots, p_{i-1}\}$ for all $i\leq t$ (in particular,
$p_1\in \{2,3\}$) \cite[Example 6.14]{hutchinson2022}.
\end{exa}

For any non-negative integer $i$, let 
\[
Z_i(A^2)=\ker(\partial_i), \ \ \ \ B_i(A^2)=\im(\partial_{i+1}). 
\]
Following Coronado and Hutchinson \cite[\S3]{C-H2022} we define $\RP(A)$ as follows.
\begin{dfn}
For a commutative ring $A$, let
\[
\RP(A):=H_0(\SL_2(A), Z_2(A^2))=H_0(\PSL_2(A), Z_2(A^2)).
\]
\end{dfn}

See \cite[\S6]{C-H2022} and \cite[\S2]{B-E--2023} for a justification of this definition. Note that 
$\RP(A)$ has a natural $\GG_A$-module structure. The inclusion \inc: $Z_2(A^2) \harr X_2(A^2)$ 
induces the map
\[
\lambda:=\inc_\ast: \RP(A)  \larr H_0(\PSL_2(A), X_2(A^2)).
\]
Let ${\pmb e_1}:={\mtx 1 0}$ and $ {\pmb e_2}:={\mtx 0 1}$ and set
\[
{\pmb\infty}:=\lan {\pmb e_1}\ran, \ \ \  {\pmb 0}:=\lan {\pmb e_2}\ran , \ \ \  
{\pmb a}:=\lan {\pmb e_1}+ a{\pmb e_2}\ran, \ \ \ a\in \aa.
\]
The group $\PSL_2(A)$ acts transitively on the generators of $X_i(A^2)$ for $i=0,1$.
We choose $({\pmb \infty})$ and $({\pmb \infty} ,{\pmb 0})$ as representatives of the orbit of 
the generators of $X_0(A^2)$ and $X_1(A^2)$, respectively. Therefore 
\[
X_0(A^2)\simeq \Ind _{\PB(A)}^{\PSL_2(A)}\z, \ \ \ \ \ \ X_1(A^2)\simeq \Ind _{\PT(A)}^{\PSL_2(A)}\z,
\]
where 
\[
\PB(A):=\stabe_{\PSL_2(A)}({\pmb \infty})=\Bigg\{\begin{pmatrix}
a & b\\
0 & a^{-1}
\end{pmatrix}:a\in \aa, b\in A\bigg\}/\mu_2(A)I_2,
\]
\[
\PT(A):=\stabe_{\PSL_2(A)}({\pmb \infty},{\pmb 0})=\Bigg\{\begin{pmatrix}
a & 0\\
0 & a^{-1}
\end{pmatrix}:a\in \aa\bigg\} /\mu_2(A)I_2.
\]
Thus  by Shapiro's lemma we have
\[
H_q(\PSL_2(A), X_0(A^2)) \simeq H_q(\PB(A),\z),
\]
\[
H_q(\PSL_2(A), X_1(A^2)) \simeq H_q(\PT(A),\z).
\]
Note that $\PT(A)\simeq \aa/\mu_2(A)$. For $a\in \aa$, let 
\[
D(a):=\overline{{\mtxx a 0 0 {a^{-1}}}} \in \PSL_2(A).
\]
The orbits of the action of  $\PSL_2(A)$ on $X_2(A^2)$ are represented by
$\lan a\ran[\ ]:=({\pmb \infty}, {\pmb 0}, {\pmb a})$, $\lan a\ran\in \GG_A$. 
The stabilizer of $({\pmb \infty}, {\pmb 0}, {\pmb a})$ is trivial. Thus
\[
X_2(A^2)\simeq \bigoplus_{\lan a\ran \in \GG_A} \Ind _{1}^{\PSL_2(A)}\z.
\]
Again by Shapiro's lemma 
\[
H_q(\PSL_2(A),X_2(A^2)) \simeq \bigoplus_{\lan a\ran\in \GG_A} H_q(1, \z)\simeq
\begin{cases}
 \z[\GG_A]  & \text{if $q=0$} \\
 0          & \text{if $q\neq  0$}
\end{cases}. 
\]
\begin{dfn}
The $\GG_A$-module  
\[
\RP_1(A):=\ker(\lambda:\RP(A) \arr \z[\GG_A])
\]
is called the {\bf refined scissors congruence group of $A$}.
\end{dfn} 

\begin{lem}
For any commutative ring $A$,  $\RP_1(A)\simeq H_1(\PSL_2(A),B_1(A^2))$.  In particular, if
$A$ satisfies the condition that $X_\bullet(A^2)$ is exact in dimension $1$, then 
$\RP_1(A)\simeq H_1(\PSL_2(A),Z_1(A^2))$.
\end{lem}
\begin{proof}
From the exact sequence $0 \arr Z_2(A^2) \overset{\inc}{\larr} X_2(A^2) \overset{\partial_2}{\larr} B_1(A^2) \arr 0$
we obtain the long exact sequence
\[
 H_1(\PSL_2(A),X_2(A^2)) \arr H_1(\PSL_2(A),B_1(A^2)) \overset{\delta}{\larr} \RP(A)  
 \overset{\lambda}{\larr} \z[\GG_A].
\]
We showed in above that $H_1(\PSL_2(A),X_2(A^2))=0$.  This implies the claim.
\end{proof}

Motivated by this lemma, we define

\begin{dfn}
For any commutative ring $A$, we define
\[
\RP_1'(A):=H_1(\PSL_2(A),Z_1(A^2)).
\]
\end{dfn}

\begin{rem}\label{RPRP'}
We always have a natural map 
\[
\RP_1(A) \arr \RP_1'(A),
\]
induced by the inclusion
$B_1(A^2)\harr Z_1(A^2)$. Clearly if $X_\bullet(A^2)$ is exact in dimension 1, i.e. 
$H_1(X_\bullet(A^2))=0$, then $\RP_1(A) = \RP_1'(A)$. 
\end{rem}

From the exact sequence 
\[
0 \arr  Z_1(A^2) \overset{\inc}{\larr} X_1(A^2) \overset{\partial_1}{\larr} B_0(A^2) \arr 0
\]
we obtain the long exact sequence
\[
H_2(\PSL_2(A),X_1(A^2)) \arr H_2(\PSL_2(A),B_0(A^2))\arr \RP_1'(A) \arr H_1(\PSL_2(A),X_1(A^2)).
\]
The right hand map is trivial (see the proof of Lemma~\ref{d21} below). 
Thus we have the exact sequence
\[
H_2(\PT(A),\z) \arr H_2(\PSL_2(A),B_0(A^2))\arr \RP_1'(A) \arr 0.
\]
Consider the inclusion $B_0(A^2) \harr X_0(A^2)$. The composition
\[
H_2(\PT(A),\z)\simeq H_2(\PSL_2(A),X_1(A^2)) \arr H_2(\PSL_2(A),B_0(A^2)) 
\]
\[
\arr 
H_2(\PSL_2(A),X_0(A^2))=H_2(\PB(A),\z)
\]
is trivial. In fact, this is the differential $d_{1,2}^1$ of the spectral sequence introduced in the 
next section (see the paragraph above Lemma~\ref{d21}).
Thus we obtain a natural map
\[
\RP_1'(A) \simeq H_2(\PSL_2(A),B_0(A^2))/H_2(\PT(A),\z) \arr H_2(\PB(A),\z),
\]
which we denote it by $\lambda_1'$:
\[
\lambda_1': \RP_1'(A) \arr H_2(\PB(A),\z).
\]
We denote the composite
\[
\RP_1(A) \arr  \RP_1'(A) \overset{\lambda_1'}{\larr} H_2(\PB(A),\z)
\]
with $\lambda_1$.

\begin{dfn}
The kernel of $\lambda_1$ is called the {\bf refined Bloch group} of $A$ and is denoted by $\RB(A)$:
\[
\RB(A):=\ker(\RP_1(A) \overset{\lambda_1}{\larr} H_2(\PB(A),\z)).
\]
Moreover, we denote the kernel of $\lambda_1'$ by $\RB'(A)$:
\[
\RB'(A):=\ker(\RP_1'(A) \overset{\lambda_1'}{\larr} H_2(\PB(A),\z)).
\]
\end{dfn}

\begin{rem}\label{IA}
This definition of $\RB(A)$ is a bit different than the one defined in \cite[\S4]{C-H2022} and 
\cite[\S4]{B-E--2023} (see Remark \ref{RB(Z)} below). But when $H_2(\Tt(A),\z)\simeq H_2(\Bb(A),\z)$ 
and $H_2(\PT(A),\z)\simeq H_2(\PB(A),\z)$, these two definitions coincide.
\end{rem}
For any commutative ring $A$, let
\[
\GW'(A):=H_0(\PSL_2(A), Z_1(A^2)).
\] 
The inclusion $\inc:Z_1(A^2) \harr X_1(A^2)$ induces the map
\[
\varepsilon:=\inc_\ast: \GW'(A) \arr H_0(\PSL_2(A), X_1(A^2))\simeq H_0(\PT(A),\z)=\z.
\]
This map is surjective, since the composite 
\[
\z[\GG_A]\simeq H_0(\PSL_2(A), X_2(A^2))\arr \GW'(A) \arr H_0(\PSL_2(A), X_1(A^2))\simeq\z
\]
is the augmentation map (which clearly is surjective). We denote the kernel of $\varepsilon$ by $I'(A)$:
\[
I'(A):=\ker(\GW'(A) \overset{\varepsilon}{\larr} \z).
\]


\begin{rem}\label{GW}
If $A$ satisfies the condition that $X_\bullet(A^2)\arr \z$ is exact in dimension $< 3$,
then $\GW'(A)$ is isomorphic to 
\[
\overline{\GW}(A):=\z[\GG_A]/\lan \Lan a\Ran\Lan 1-a\Ran:a\in \WW_A\ran
\]
(see \cite[\S2]{B-E--2023}), where $\WW_A:=\{a\in A: a(a-1)\in \aa\}$.
When $A$ is a field of odd characteristic or a local ring 
with sufficiently large residue field $k$, where $\char(k)\neq 2$, this latter group coincides 
with the classical Grothendieck-Witt group of symmetric bilinear forms over $A$ (see 
\cite[Lemma 1.1, Chap. 4]{HM1973}, \cite[page 203]{mazzoleni2005}, \cite[Proposition 4.6]{{suslin1985}}). 
\end{rem}

\section{The homology of projective special linear group of degree 2}\label{sec2}

Let $L_\bullet(A^2)$ be the complex
\begin{equation}\label{comp1}
0 \arr Z_1(A^2)  \overset{\inc}{\arr} X_1(A^2)  \overset{\partial_1}{\arr} X_0(A^2)  \arr 0.
\end{equation}
Let $D_{\bullet,\bullet}$ be the double complex $F_\bullet\otimes_{\PSL_2(A)} L_\bullet(A^2)$,
where $F_\bullet \arr \z$ is a projective resolution of $\z$ over $\PSL_2(A)$. From 
$D_{\bullet,\bullet}$ we obtain the first quadrant spectral sequence
\[
E^1_{p.q}=H_q(\PSL_2(A),L_p(A^2))\Longrightarrow H_{p+q}(\PSL_2(A),L_\bullet(A^2))
\]
(see \cite[\S5, Chap. VII]{brown1994}).
Note that in this spectral sequence (and any other spectral sequence) our notation for the 
differential $d_{p,q}^r$ is as follows:
\[
d_{p,q}^r: E_{p,q}^r \arr E_{p-r,q+r-1}^r.
\]
In this spectral sequence, $H_n(\PSL_2(A),L_\bullet(A^2))$ is the $n$-th hyperhomology of $\PSL_2(A)$ 
with coefficients in the complex $L_\bullet(A^2)$ (see \cite[\S5.7]{weibel1994}, 
\cite[1.3,\S1, Chap. III]{brown1994}). Thus $H_n(\PSL_2(A),L_\bullet(A^2))$ is the $n$-th homology of the 
total complex of the double complex $D_{\bullet,\bullet}$. 

\begin{lem}
If $A$ is a $\GE_2$-ring, then for any $n$
\[
H_n(\PSL_2(A),L_\bullet(A^2))\simeq H_n(\PSL_2(A),\z).
\]
\end{lem}
\begin{proof}
Since $A$ is a $\GE_2$-ring, by Proposition \ref{GE2C} we have $H_0(L_\bullet(A^2))\simeq \z$. 
Moreover, by definition, $H_i(L_\bullet(A^2))=0$ for any $i\geq 1$. Now the claim follows from
\cite[Lemma 4.6]{hutchinson-2013}.
\end{proof}

In our calculations, for the resolution $F_\bullet \arr \z$ we usually use either the bar resolution 
$B_\bullet(\PSL_2(A))\arr \z$ or the standard resolution $C_\bullet(\PSL_2(A))\arr \z$ 
\cite[Chap.I, \S 5]{brown1994}. Note that we have the natural morphisms of complexes 
$B_\bullet(\PSL_2(A)) \arr C_\bullet(\PSL_2(A))$ and 
$C_\bullet(\PSL_2(A)) \arr B_\bullet(\PSL_2(A))$ given by
\[
B_n(\PSL_2(A)) \arr C_n(\PSL_2(A)),  \ \ \ [h_1|h_2|\cdots|h_n]\mapsto (1,h_1,h_1h_2, 
\dots, h_1h_2\cdots h_n),
\]
\[
C_n(\PSL_2(A)) \arr B_n(\PSL_2(A)), \ \ \ (g_0,\dots, g_n) \mapsto  g_0[g_0^{-1}g_1|
g_1^{-1}g_2|\cdots |g_{n-1}^{-1}g_n],
\]
which induce the identity on the homology of $\PSL_2(A)$ with any coefficients.

We have seen that
\[
E_{0,q}^1 \simeq H_q(\PB(A),\z), \  \  \ 
E_{1,q}^1 \simeq H_q(\PT(A),\z).
\]
In particular, $E_{0,0}^1\simeq\z\simeq E_{1,0}^1$. Moreover 
\[
d_{1,q}^1=H_q(\sigma)-H_q(\inc),
\]
where $\sigma: \PT(A) \arr \PB(A)$ is given by $\sigma(D(a))= wD(a) w^{-1}=D(a)^{-1}$ for 
$w=\overline{{\mtxx 0 1 {-1} 0}}$. This easily implies that $d_{1,0}^1$ is trivial and 
\[
d_{1,1}^1:\PT(A)=H_1(\PT(A),\z)\arr H_1(\PB(A),\z)=\PB(A)/[\PB(A),\PB(A)]
\]
is given by $D(a)\mt \overline{D(a)}^{-2}$. Moreover, $d_{1,2}^1$ is trivial. In fact, if 
under the isomorphism $\PT(A)\wedge \PT(A)\simeq H_2(\PT(A),\z)$, the image of $D(a)\wedge D(b)$
is ${\bf c}(D(a), D(b))$, then 
\[
d_{1,2}^1:H_2(\PT(A),\z)\arr H_2(\PB(A),\z),
\]
is given by
\[
d_{1,2}^1({\bf c}(D(a), D(b))) = {\bf c}(D(a)^{-1}, D(b)^{-1}) - {\bf c}(D(a), D(b))=0.
\]

\begin{lem}\label{d21}
The differential $d^1_{2,1}$ is trivial. In particular, $E_{2,1}^2=\RP_1'(A)$
and $E_{1,1}^2\simeq \widetilde{\mu}_4(A)$.
\end{lem}
\begin{proof}
The natural map $\SL_2(A) \arr \PSL_2(A)$ induces the morphism of spectral sequences
\[
\begin{tikzcd}
{E'}_{p,q}^1\ar[r,Rightarrow] \ar[d]& H_{p+q}(\SL_2(A),L_\bullet(A^2))\ar[d]\\
E_{p,q}^1\ar[r,Rightarrow]& H_{p+q}(\PSL_2(A),L_\bullet(A^2)),
\end{tikzcd}
\]
where
\[
{E'}^1_{p.q}=H_q(\SL_2(A),L_p(A^2)).
\]
The spectral sequence ${E'}_{\bullet,\bullet}$ has been studied in \cite{B-E-2023} and \cite{B-E--2023}.
From the above morphism we obtain the commutative diagram
\begin{equation*}
\begin{tikzcd}
H_1(\SL_2(A),Z_1(A^2))\ar[r,"{d'}^1_{2,1}"] \ar[d,"u"]& H_1(\Tt(A),\z) 
\ar[r,"{d'}^1_{1,1}"] \ar[d,"u'"] & H_1(\Bb(A),\z) \ar[d,"u''"] \\
H_1(\PSL_2(A),Z_1(A^2))\ar[r,"d^1_{2,1}"]& H_1(\PT(A),\z) \ar[r,"d^1_{1,1}"] & H_1(\PB(A),\z),
\end{tikzcd}	
\end{equation*}
where 
\[
\Bb(A):=\Bigg\{\begin{pmatrix}
a & b\\
0 & a^{-1}
\end{pmatrix}:a\in \aa, b\in A\bigg\},
\ \ \ \ 
\Tt(A):=\Bigg\{\begin{pmatrix}
a & 0\\
0 & a^{-1}
\end{pmatrix}:a\in \aa\bigg\}.
\]
From the Lyndon/Hochschild-Serre spectral sequence of the central extension
\[
1\arr \mu_2(A) \arr \SL_2(A) \arr \PSL_2(A)\arr 1,
\]
it follows that $u$ is surjective (See \cite[Corollary 6.4, Chap. VII]{brown1994}). 
Using the fact that $H_1(G,\z)\simeq G/[G,G]$, $G$ a group, 
we see that $u'$ and $u''$ are surjective.

The differential ${d'}^1_{1,1}: H_1(\Tt(A),\z) \arr H_1(\Bb(A),\z)$ is induced by the
map $\Tt(A) \arr \Bb(A)$, $X\mapsto X^{-2}$ \cite[\S2]{B-E-2023}. So
\[
\ker({d'}^1_{1,1})=\mu_2(A)I_2
\]
(see the last paragraph of \cite[\S2]{B-E-2023}). Since 
\[
u'(\ker({d'}^1_{1,1}))=u'(\mu_2(A))=1, \ \ \ \ \im({d'}^1_{2,1})\subseteq \ker({d'}^1_{1,1})=\mu_2(A)I_2,
\]
we have $u'\circ {d'}^1_{2,1}=0$. Therefore, from the above commutative diagram we get $d_{2,1}^1=0$.
\end{proof}

\begin{lem}\label{coincide1}
The differential $d_{2,1}^2:\RP_1'(A) \arr H_2(\PB(A),\z)$ coincides with the map $\lambda_1'$. 
In particular, $E_{2,1}^\infty\simeq E_{2,1}^3=\RB'(A)$.
\end{lem}
\begin{proof}
It is sufficient to show that $d^2_{2,1}$ can replace $\lambda'_1$ in the following diagram 
with exact rows
\[
\begin{tikzcd}
& H_2(\PT(A),\z)\ar[r]\ar[d] & H_2(\PSL_2(A),B_0(A^2)) \ar[r,"\delta"]\ar[d,"\inc_*"] & 
\RP'_1(A) \ar[r]\ar[d,"\lambda'_1"] & 0\\
0 \ar[r] & 0 \ar[r] & H_2(\PB(A),\z)\ar[r,equal] & H_2(\PB(A),\z) \ar[r] & 0.
\end{tikzcd}
\]
The map $\delta:H_2(\PSL_2(A),B_0(A^2))\arr \RP'_1(A)=H_1(\PSL_2(A),Z_1(A^2))$ is the connecting 
homomorphism of the long exact sequence associated to the short exact sequence
\[
0\arr Z_1(A^2) \overset{\inc}{\arr} X_1(A^2) \overset{\partial_1}{\arr} B_0(A^2)\arr 0.
\]
Let $\delta(\overline{x})\in\RP_1'(A)$, where $\overline{x}\in H_2(\PSL_2(A),B_0(A^2))$ is represented 
by a cycle $x\in F_2\otimes_{\PSL_2(A)} B_0(A^2)$. Here $F_{\bullet}$ is a projective resolution of 
$\z$ over $\PSL_2(A)$ (with differentials $d_{\bullet}$). Consider the diagram
\[
\begin{tikzcd}
&  F_2 \otimes_{\PSL_2(A)} X_1(A^2) \ar[r, "\id_{F_2}\otimes\partial_1"]\ar[d, "d_2\otimes\id_{X_1}"] & 
F_2 \otimes_{\PSL_2(A)} B_0(A^2) \\
F_1 \otimes_{\PSL_2(A)} Z_1(A^2) \ar[r, "\id_{F_1}\otimes\inc"] & F_1 \otimes_{\PSL_2(A)} X_1(A^2). &
\end{tikzcd}
\]
Observe that
\[
x=(\id_{F_2}\otimes\partial_1)(y)
\]
for some $y\in F_2\otimes_{\PSL_2(A)}X_1(A^2)$. For the element
$(d_2\otimes\id_{X_1})(y)\in F_1\otimes_{\PSL_2(A)} X_1(A^2)$, we have 
\[
(\id_{F_1}\otimes \partial_1)(d_2\otimes\id_{X_1})(y)=0.
\]
So there is $z\in F_1\otimes_{\PSL_2(A)} Z_1(A^2)$ such that
\[
(d_2\otimes\id_{X_1})(y)=(\id_{F_1}\otimes\inc)(z). 
\]
Note that $\delta(\overline{x})=\overline{z} \in H_1(\PSL_2(A),B_0(A^2))$. Moreover, observe that
\[
\lambda_1'(\delta(\overline{x}))=\lambda_1'(\overline{z})=\overline{(\id_{F_2}\otimes \partial_1)(x)}.
\]
Now applying $d^2_{2,1}$ to the element $\overline{z}=\delta(\overline{x})$ requires going through 
the diagram
\[
\begin{tikzcd}
F_2\otimes_{\PSL_2(A)}B_0(A^2) & F_2\otimes_{\PSL_2(A)} X_1(A^2) 
\ar[l,"\id_{F_2}\otimes\partial_1"']\ar[d,"d_2\otimes\id_{X_1}"] & \\
& F_1\otimes_{\PSL_2(A)}X_1(A^2) & F_1\otimes_{\PSL_2(A)}Z_1(A^2) 
\ar[l,"\id_{F_1}\otimes\inc"']
\end{tikzcd}
\]
upward. As in above this will end up with the element $\overline{(\id_{F_2}\otimes\partial_1)(x)}$.
Therefore $d^2_{2,1}$ coincides with $\lambda_1'$.
\end{proof}

It is straightforward to check that
\[
E_{0,1}^1\simeq \widetilde{\GG}_A \oplus A_\aa,\ \ \ \ \ 
E_{1,2}^1\simeq \PT(A)\wedge\PT(A)\simeq\frac{\aa \wedge \aa}{\mu_2(A) \wedge\aa},
\]
where $\widetilde{\GG}_A:=\GG_A/\mu_2(A)$.

\begin{lem}\label{GG}
The group $E_{1,2}^2$ is a quotient of the $2$-torsion group 
$\displaystyle\frac{\GG_A\wedge \GG_A}{\mu_2(A)\wedge \GG_A}$.
\end{lem}
\begin{proof}
We have seen that $E_{1,2}^1\simeq H_2(\PT(A),\z) \simeq \displaystyle\frac{\aa \wedge \aa}{\aa \wedge \mu_2(A)}$.
Consider the differential 
\[
d_{2,2}^1: H_2(\PSL_2(A),Z_1(A^2))\arr H_2(\PT(A),\z)\simeq \PT(A) \wedge \PT(A). 
\]
It is straightforward to check that 
\[
([D(a)|D(b)]-[D(b)|D(a)])\otimes Y \in B_2(\PSL_2(A))\otimes_{\PSL_2(A)} Z_1(A^2)
\]
is a cycle
and 
\[
d_{2,2}^1(\overline{([D(a)|D(b)]-[D(b)|D(a)])\otimes Y})=2(d(a)\wedge D(b)),
\]
where $Y=({\pmb\infty},{\pmb 0})+({\pmb 0},{\pmb \infty})$. 
Thus $E_{1,2}^2$ is a quotient of
\begin{align*}
H_2(\PT(A),\z)/2 & 
\simeq \displaystyle \frac{\aa\wedge \aa}{2(\aa\wedge \aa)+(\aa\wedge \mu_2(A))} \\
& \simeq \frac{(\aa\wedge \aa)/2(\aa\wedge \aa)}{\Big(2(\aa\wedge \aa)+(\aa\wedge \mu_2(A))\Big)/2(\aa\wedge \aa)}.
\end{align*}
The rest follows from the following lemma.
\end{proof}

\begin{lem}\label{wedge}
Let $\Aa$ be an abelian group. Then for any positive integers $m$ and $n$, we have the isomorphism
\[
\begin{array}{c}
(\bigwedge_\z^n \Aa)/m  \simeq \bigwedge_\z^n (\Aa/m)\simeq \bigwedge_{\z/m}^n (\Aa/m).
\end{array}
\]
\end{lem}
\begin{proof}
The maps  $(\bigwedge_\z^n \Aa)/m  \arr \bigwedge_\z^n (\Aa/m)$, given by 
$\overline{a_1\wedge \dots \wedge a_n}\mapsto \overline{a_1}\wedge \dots \wedge \overline{a_1}$, and
$\bigwedge_\z^n (\Aa/m)\arr (\bigwedge_\z^n \Aa)/m$, given by 
$\overline{a_1}\wedge \dots \wedge \overline{a_1}\mapsto \overline{a_1\wedge \dots \wedge a_n}$, are 
well-defined and one is the inverse of the other. The other isomorphism follows naturally.
\end{proof}

In general we think that $E_{1,2}^2$ always is trivial! We prove this in the following
interesting case.

\begin{lem}\label{surj1}
If $-1\in \aa^2$, then the differential $d^1_{2,2}$ is surjective. In particular, $E_{1,2}^2=0$.
\end{lem}
\begin{proof}
In this proof, for simplicity, we denote the matrix  $D(a)$ with $\hat{a}$.
Let $i\in \aa$ such that $i^2=-1$. 
For $a\in A^{\times}$, denote  $({\pmb \infty}, {\pmb 0}, {\pmb a})\in X_2(A^2)$ 
by $X_a$. Let $Y:=({\pmb\infty},{\pmb 0})+({\pmb 0},{\pmb \infty})\in Z_1(A^2)$. 
For $a,b\in \aa$, let 
\[
\lambda(a,b)\in H_2(\PSL_2(A),Z_1(A))=H_2(B_\bullet(\PSL_2(A))\otimes_{\PSL_2(A)} Z_1(A))
\]
 be represented by $X_{a,b}\in B_2(\PSL_2(A))\otimes_{\PSL_2(A)} Z_1(A^2)$, where
\begin{align*}
X_{a,b}&:=([\hat{a}|\hat{b}]+[\hat{i}|\widehat{ab}]-[\hat{i}|\hat{a}]-[\hat{i}|\hat{b}]+[w|\widehat{iab}]
-[w|\widehat{ia}]-[w|\widehat{ib}]+[w|\hat{i}])\otimes Y\\
&+[w\widehat{iab}|w\widehat{iab}]\otimes \partial_2(X_{ab})-[w\widehat{ia}|w\widehat{ia}]\otimes 
\partial_2(X_a)-[w\widehat{ib}|w\widehat{ib}]\otimes\partial_2(X_b)\\
&+ [w\hat{i}|w\hat{i}]\otimes \partial_2(X_1).
\end{align*}
Recall that $w=\overline{\mtxx{0}{1}{-1}{0}}$. (Observe that we need the condition $i^2=-1$ to show that 
$X_{a,b}$ is a cycle.) Now 
\[
d^1_{2,2}(\lambda(a,b))\in H_2(\PSL_2(A),X_1(A^2))=H_2(B_\bullet(\PSL_2(A))\otimes_{\PSL_2(A)} X_1(A))
\]
is represented by the cycle $Y_{a,b}\in B_2(\PSL_2(A))\otimes_{\PSL_2(A)} X_1(A)$, where
\begin{align*}
Y_{a,b}&:=(w+1)([\hat{a}|\hat{b}]+[\hat{i}|\widehat{ab}]-[\hat{i}|\hat{a}]-
[\hat{i}|\hat{b}]+[w|\widehat{iab}]-[w|\widehat{ia}]-[w|\widehat{ib}]+[w|\hat{i}])\otimes ({\pmb \infty}, {\pmb 0})\\
&+(g_{ab}^{-1}-h_{ab}^{-1}+1)([w\widehat{iab}|w\widehat{iab}])\otimes ({\pmb \infty}, {\pmb 0})
-(g_{a}^{-1}-h_{a}^{-1}+1)([w\widehat{ia}|w\widehat{ia}])\otimes ({\pmb \infty}, {\pmb 0})\\
&-(g_{b}^{-1}-h_{b}^{-1}+1)([w\widehat{ib}|w\widehat{ib}])\otimes ({\pmb \infty}, {\pmb 0})
+(g_1^{-1}-h_1^{-1}+1)([w\hat{i}|w\hat{i}])\otimes ({\pmb \infty}, {\pmb 0}),
\end{align*}
with
$g_x:=\overline{\begin{pmatrix}
0 & 1\\
-1 & x
\end{pmatrix}}$ and 
$h_x:=\overline{\begin{pmatrix}
1 & x^{-1}\\
0 & 1
\end{pmatrix}}$.
The morphisms
\[
B_\bullet(\PSL_2(A))\otimes_{\PSL_2(A)} X_1(A) \arr B_\bullet(\PSL_2(A))\otimes_{\PT(A)} \z 
\arr C_\bullet(\PSL_2(A))\otimes_{\PT(A)} \z,
\]
where on degree $n$ is given by
\[
[g_1|\cdots|g_n]\otimes ({\pmb\infty}, {\pmb 0}) \mt [g_1|\cdots|g_n]\otimes 1 
\mt \otimes (1, g_1, \dots, g_1\cdots g_n)\otimes 1,
\]
induce the isomorphisms

\begin{align*}
H_2(B_\bullet(\PSL_2(A))\otimes_{\PSL_2(A)} X_1(A)) & \simeq H_2(B_\bullet(\PSL_2(A))\otimes_{\PT(A)} \z)\\
& \\
&\simeq  H_2(C_\bullet(\PSL_2(A))\otimes_{\PT(A)} \z).
\end{align*}
Following these maps we see that $d^1_{2,2}(\lambda(a,b))=\overline{Y_{a,b}'} \in H_2(C_\bullet(\PSL_2(A))\otimes_{\PT(A)} \z)$,
where $Y_{a,b}'\in C_2(\PSL_2(A))\otimes_{\PT(A)} \z$ is the cycle
\begin{align*}
Y_{a,b}'&:=((w,w\hat{a},w\widehat{ab})+(w,w\hat{i},w\widehat{iab})-(w,w\hat{i},w\widehat{ia})-(w,w\hat{i},w\widehat{ib})\\
&+(w,1,\widehat{iab})-(w,1,\widehat{ia})-(w,1,\widehat{ib})+(w,1,\hat{i})\\
&+(1,\hat{a},\widehat{ab})+(1,\hat{i},\widehat{iab})+(1,\hat{i},\widehat{ia})+(1,\hat{i},\widehat{ib})\\
&+(1,w,w\widehat{iab})+(1,w,w\widehat{ia})+(1,w,w\widehat{ib})+(1,w,w\hat{i}))\otimes 1\\
&+((g^{-1}_{ab},g^{-1}_{ab}w\widehat{iab},g^{-1}_{ab})
-(h^{-1}_{ab},h^{-1}_{ab}w\widehat{iab},h^{-1}_{ab})+(1,w\widehat{iab},1))\otimes 1\\
&-((g^{-1}_{a},g^{-1}_{a}w\widehat{ia},g^{-1}_{a})-(h^{-1}_{a},h^{-1}_{a}w\widehat{ia},h^{-1}_{a})
+(1,w\widehat{ia},1))\otimes 1\\
&-((g^{-1}_{b},g^{-1}_{b}w\widehat{ib},g^{-1}_{b}))-(h^{-1}_{b},h^{-1}_{b}w\widehat{ib},h^{-1}_{b})
+(1,w\widehat{ib},1)\otimes 1\\
&+((g^{-1}_1,g^{-1}_1w\hat{i},g^{-1}_1)-(h^{-1}_1,h^{-1}_1w\hat{i},h^{-1}_1)+(1,w\hat{i},1))\otimes 1.
\end{align*}
We know that
\[
H_2(C_\bullet(\PSL_2(A))\otimes_{\PT(A)} \z)\simeq H_2(C_\bullet(\PT(A))\otimes_{\PT(A)} \z).
\]
We want to find a representative of $d^1_{2,2}(\lambda(a,b))$ in $C_2(\PT(A))\otimes_{\PT(A)}\z$. Let  
\[
s:\PT(A)\backslash\PSL_2(A)\to \PSL_2(A)
\]
be any (set-theoretic) section of the canonical projection $\pi:\PSL_2(A)\to \PT(A)\backslash\PSL_2(A)$, 
$g \mapsto \PT(A)g$. For $g\in \PSL_2(A)$, set
\[
\overline{g}:=g(s\circ \pi(g))^{-1}=gs(gH)^{-1}.
\]
Then by Lemma \ref{G-H} below, the morphism
\[
C_\bullet(\PSL_2(A))\otimes_{\PT(A)} \z \overset{s_\bullet}{\larr} C_\bullet(\PT(A))\otimes_{\PT(A)}\z, 
\ \ \ (g_0,\dots,g_n)\otimes 1 \mt (\overline{g_0},\dots,\overline{g_n})\otimes 1
\]
induces the above isomorphism. Choose a section 
\[
s:\PSL_2(A)\backslash \PT(A)\to \PSL_2(A)
\]
such that
\[
s(\PT(A)\begin{pmatrix}
a&b\\
c&d
\end{pmatrix})
=
\begin{pmatrix}
1&a^{-1}b\\
ac&ad
\end{pmatrix}, \ \ \ \ \ \  \text{if  $a\in \aa$,}\\
\]
and
\[
s(\PT(A)\begin{pmatrix}
a&b\\
c&d
\end{pmatrix})
=
\begin{pmatrix}
0&1\\
-1&bd
\end{pmatrix}, \ \ \ \ \ \ \text{if $a=0$.}
\]
It is straightforward to check that $d^1_{2,2}(\lambda(a,b))=\overline{T_{a,b}} \in H_2(C_\bullet(\PT(A))\otimes_{\PT(A)}\z)$, 
where
\begin{align*}
T_{a,b}&:=\bigg((1,{\hat{a}^{-1}},{(\widehat{ab})^{-1}})+(1,{\hat{i}},{\hat{i}(\widehat{ab})^{-1}})
-(1,{\hat{i}},{\hat{i}\hat{a}^{-1}})-(1,{\hat{i}},{\hat{i}\hat{b}^{-1}})\\
&+(1,\hat{a},\widehat{ab})+(1,\hat{i},\widehat{iab})-(1,\hat{i},\widehat{ia})-(1,\hat{i},\widehat{ib})+(1,{\hat{i}(\widehat{ab})^{-1}},1)\\
&-(1,{\hat{i}\hat{a}^{-1}},1)-(1,{\hat{i}\hat{b}^{-1}},1)+(1,\hat{i},1)\bigg)\otimes 1.
\end{align*}
In $H_2(B_{\bullet}(\PT(A))\otimes_{\PT(A)} \z)$ this elements corresponds to $d^1_{2,2}({\lambda}(a,b))=\overline{S_{a,b}}$,
where
\begin{align*}
S_{a,b}&:=([{\hat{a}^{-1}}|{\hat{b}^{-1}}]+[\hat{i}|{(\widehat{ab})^{-1}}]-[\hat{i}|{\hat{a}^{-1}}]-[\hat{i}|{\hat{b}^{-1}}]
+[\hat{a}|\hat{b}]+[\hat{i}|\widehat{ab}]-[\hat{i}|\hat{a}]-[\hat{i}|\hat{b}]\\
&+[{\hat{i}\hat{a}^{-1}\hat{b}^{-1}}|\widehat{iab}]-[{\hat{i}\widehat{a}^{-1}}|\widehat{ia}]
-[{\hat{i}\hat{b}^{-1}}|\widehat{ib}]+[\hat{i}|\hat{i}])\otimes 1.
\end{align*}
Now by adding the null element (in $B_2(\PT(A))\otimes_{\PT(A)} \z$)
\[
d_3(-[\hat{i}|{\hat{a}^{-1}}|{\hat{b}^{-1}}]-[\hat{i}|\hat{a}|\hat{b}]+[{\hat{i}\hat{a}^{-1}}|{\hat{b}^{-1}}|\widehat{iab}]
-[{\hat{b}^{-1}}|\hat{b}|\widehat{ia}]-[\hat{i}|{\hat{b}^{-1}}|\widehat{ib}]+[{\hat{b}^{-1}}|\hat{b}|\hat{i}])\otimes 1,
\]
it is easy to see that
\[
d^1_{2,2}(\lambda(a,b))=\overline{([\widehat{ia}|\hat{b}]-[\hat{b}|\widehat{ia}]-[\hat{i}|\hat{b}]+[\hat{b}|\hat{i}])\otimes 1}
\in H_2(B_{\bullet}(\PT(A))\otimes_{\PT(A)} \z).
\]
Note that
\[
H_2(B_{\bullet}(\PT(A))\otimes_{\PT(A)} \z)=H_2(\PT(A),\z)\simeq \PT(A) \wedge \PT(A).
\]
In $\PT(A) \wedge \PT(A)$, we have 
\[
d^1_{2,2}(\lambda(a,b))=\widehat{ia}\wedge \hat{b}-\hat{i}\wedge \hat{b}=\hat{a}\wedge \hat{b}.
\]
This shows that $d^1_{2,2}$ is surjective.
\end{proof}

\begin{lem}\label{G-H}
Let $H$ be a subgroup of $G$ an $H\backslash G$ the set of right cosets of $H$ in $G$. 
Let $s:H\backslash G \arr G$ be any (set theoretic)
section of the natural map (of sets) $\pi:G \arr H\backslash G$, $g \mt Hg$. For $g \in G$, let 
$\overline{g}:=g\Big(s\circ\pi(g)\Big)^{-1}$. Then the map $\phi_n:C_n(G) \arr C_n(H)$
given by $(g_0, \dots, g_n) \mt (\overline{g_0}, \dots, \overline{g_n})$ induces a left 
$H$-morphism of the standard complexes $\phi_\bullet:C_\bullet(G) \arr C_\bullet(H)$ and, 
for any $k$, the homomorphism 
\[
H_k(H,\z)= H_k(C_\bullet(G)\otimes_H \z) \arr H_k(C_\bullet(H)\otimes_H \z) = H_k(H,\z)
\]
coincides with the identity map $\id_{H_k(H,\z)}$.
\end{lem}
\begin{proof}
It is straightforward to check that $\phi_\bullet$ is a morphism of complexes. Moreover,
note that for any $h\in H$, 
\begin{align*}
\phi_n(h(g_0,\dots,g_n))&=(\overline{hg_0},\dots,\overline{hg_n})\\
&=((hg_0)s(Hhg_0)^{-1}, \dots, (hg_n)s(Hhg_n)^{-1})\\
&=(hg_0s(Hg_0)^{-1}, \dots, hg_0s(Hg_n)^{-1})\\
&=h\phi_n(g_0,\dots,g_n).
\end{align*}
Thus $\phi_\bullet$ is a $H$-morphism. The inclusion $\inc: H \arr G$, induces the morphism 
\[
\inc_\bullet: C_\bullet(H) \arr C_\bullet(G),
\]
which induces the homomorphism
\[
H_n(H,\z)=H_n(C_\bullet(H)\otimes_H \z) \arr H_n(C_\bullet(G)\otimes_H \z)= H_n(H,\z).
\]
Let study the compositions $(\inc_\bullet\circ\phi_\bullet)\otimes_H \id_\z$ and 
$(\phi_\bullet\circ\inc_\bullet)\otimes_H \id_\z$. If $h'=s(H)\in H$, then the composition
\[
C_n(H)\otimes_H \z \overset{\inc_n\otimes \id_\z}{\larr} C_n(G)\otimes_H \z \overset{\phi_n\otimes 
\id_\z}{\larr} C_n(H)\otimes_H \z
\]
is given by
\[
(h_0,\dots,h_n)\otimes 1 \mapsto (\overline{h_0},\dots,\overline{h_n})\otimes 1
=(h_0{h'}^{-1}, \dots, h_n{h'}^{-1})\otimes 1=(h_0, \dots, h_n)\otimes 1.
\]
Thus 
\[
H_n(\phi_\bullet)\circ H_n(\inc_\bullet)=\id_{H_n(H,\z)}. 
\]
On the other hand, the composition
\[
C_n(G)\otimes_H \z \overset{\phi_n\otimes \id_\z}{\larr} C_n(H)\otimes_H \z \overset{\inc_n\otimes 
\id_\z}{\larr} C_n(G)\otimes_H \z
\]
is given by
\[
(g_0,\dots,g_n)\otimes 1 \mapsto (\overline{g_0},\dots,\overline{g_n})\otimes 1.
\]
It is straightforward to check that the maps
\[
s_{-1}=0: C_{-1}(G)=0 \arr C_0(G), 
\]
\[
s_n:C_n(G)\arr C_{n+1}(G), \ \ \ (g_0,\dots,g_n) \mapsto (1,\overline{g_0},\dots,\overline{g_n})-(1, g_0,\dots,g_n)
\]
define a homotopy between  $(\inc_\bullet\otimes \phi_\bullet)\otimes \id_\z$ and the identity morphism
$\id_\bullet:C_\bullet(G)\otimes_H \z \arr C_\bullet(G)\otimes_H \z$. Therefore, 
\[
H_n(\inc_\bullet)\circ H_n(\phi_\bullet)=\id_{H_n(H,\z)}. 
\]
This proves our claim.
\end{proof}

\begin{rem}
From the central extension $ 1 \arr \mu_2(A) \arr \SL_2(A) \arr \PSL_2(A) \arr 1$ we obtain the 
Lyndon/Hochschild-Serre spectral sequence
\[
\EE_{p,q}^2=H_p(\PSL_2(A), H_q(\mu_2(A), Z_1(A^2)))\Rightarrow H_{p+q}(\SL_2(A), Z_1(A^2)).
\]
Since the action of $\mu_2(A)$ on $Z_1(A^2)$ is trivial, from this spectral sequence we obtain the 
exact sequence
\[
\mu_2(A)\otimes_\z \RP_1'(A)\arr H_2(\SL_2(A), Z_1(A^2)) \arr H_2(\PSL_2(A), Z_1(A^2)) \arr 
\mu_2(A)\otimes_\z \GW'(A) 
\]
\[
\arr H_1(\SL_2(A), Z_1(A^2)) \arr \RP_1'(A) \arr 0.
\]
Now the element $\lambda(a,b)\in H_2(\PSL_2(A), Z_1(A^2))$, from the proof of Lemma \ref{surj1}, 
maps to the element
\[
(-1)\otimes(\Lan ab\Ran-\Lan a\Ran -\Lan b\Ran)\in \mu_2(A)\otimes_\z \GW'(A),
\]
where 
\[
\Lan a\Ran:=\lan a\ran-1=\partial_2(({\pmb \infty},{\pmb 0},{\pmb a})-({\pmb\infty},{\pmb 0},{\pmb 1})).
\]
To construct $\lambda(a,b)$ we used the assumption $-1 \in \aa^2$.
We believe that an element $\lambda(a,b)$ with the above condition exists over any ring. 
\end{rem}

Let $\Aa$ be an abelian group. Let $\sigma_1:\tors(\Aa,\Aa)\arr \tors(\Aa,\Aa)$
be obtained by interchanging the copies of $\Aa$. This map is induced by
the involution $\Aa\otimes_\z \Aa \arr \Aa\otimes_\z \Aa$, $a\otimes b\mapsto -b\otimes a$
\cite[\S2]{mmo2022}.
Let $\Sigma_2'=\{1, \sigma'\}$ be the symmetric group of order 2. Consider the following action of 
$\Sigma_2'$ on $\tors(\Aa,\Aa)$:
\[
(\sigma', x)\mapsto -\sigma_1(x).
\]

To study $E_{0,3}^2$ we need the following well-known fact. 

\begin{prp}\label{H3A}
For any abelian group $\Aa$ we have the exact sequence
\[
\begin{array}{c}
0 \arr \bigwedge_\z^3 \Aa\arr H_3(\Aa,\z) \arr \tors(\Aa,\Aa)^{\Sigma_2'}\arr 0.
\end{array}
\]
\end{prp}
\begin{proof}
See \cite[Lemma 5.5]{suslin1991} or \cite[Section~\S6]{breen1999}.
\end{proof}
Now let study the differential $d^1_{1,3}: H_3(\PT(A),\z) \arr H_3(\PB(A),\z)$. Let
\[
\PN(A):= \Bigg\{\begin{pmatrix}
b & x\\
0 & b
\end{pmatrix}:b\in \mu_2(A), x\in A\bigg\}/\mu_2(A)I_2\se \PB(A).
\]
Note that $\PN(A)\simeq A$. From the split extension 
$1 \arr \PN(A) \arr \PB(A) \arr \PT(A)\arr 1$, we obtain the decomposition 
\[
H_n(\PB(A), \z)\simeq H_n(\PT(A), \z)\oplus \Aa_n,
\]
where $\Aa_n=H_n(\PB(A),\PT(A); \z)$. By Proposition \ref{H3A}, $H_3(\PT(A),\z)$ sits in the exact sequence
\[
\begin{array}{c}
0 \arr \bigwedge_\z^3 \PT(A)\arr H_3(\PT(A),\z) \arr \tors(\widetilde{\mu}(A),\widetilde{\mu}(A))^{\Sigma_2'}\arr 0.
\end{array}
\]
It is straightforward to show that $d_{1,3}^1\mid_{\bigwedge_\z^ 3 \PT(A)}$ coincides with
multiplication by $2$ and $d_{1,3}^1\mid_{\tors(\widetilde{\mu}(A),\widetilde{\mu}(A))^{\Sigma_2'}}=0$ (see 
\cite[Corollary 1.4]{B-E-2025}). Now it follows from this that 
\[
E_{0,3}^2\simeq \LL\oplus \Aa_3,
\]
where (again by 
By Proposition \ref{H3A} and the Snake lemma) $\LL$ sits in the exact sequence
\[
\begin{array}{c}
(\bigwedge_\z^3 \PT(A))/2\arr \LL \arr \tors(\widetilde{\mu}(A),\widetilde{\mu}(A))^{\Sigma_2'}\arr 0.
\end{array}
\]
Observe that if $A$ is a domain, then $\mu(A)$ is direct limit of its finite cyclic subgroups. This implies that 
\begin{equation}\label{tor}
\tors(\widetilde{\mu}(A),\widetilde{\mu}(A))^{\Sigma_2'}=\tors(\widetilde{\mu}(A),\widetilde{\mu}(A)).
\end{equation}
In fact, if $\mu(A)$ is finite, then we may assume that 
\[
\tors(\widetilde{\mu}(A),\widetilde{\mu}(A))=\widetilde{\mu}(A)\otimes_\z \widetilde{\mu}(A)\simeq \widetilde{\mu}(A)
\]
(see \cite[Corollary 6.3.15]{vermani2003}). In this case (\ref{tor}) follows easily by the action defined before
Proposition \ref{H3A}. If $\mu(A)$ is infinite, then it is straightforward to check that
\begin{align*}
\tors(\widetilde{\mu}(A),\widetilde{\mu}(A)) 
\simeq \underset{\larr}{\lim}\tors(\widetilde{\mu}_{2n}(A),\widetilde{\mu}_{2n}(A))
 \simeq \underset{\larr}{\lim}\Big(\widetilde{\mu}_{2n}(A)\otimes_\z\widetilde{\mu}_{2n}(A)\Big) \simeq \widetilde{\mu}(A),
\end{align*}
which again (\ref{tor}) follows easily.

We usually will work with rings that have the property that $\Aa_3=0$, i.e. $H_3(\PT(A),\z)\simeq H_3(\PB(A),\z)$.
The natural map
\[
H_k(\Tt(A),\z)\arr H_k(\Bb(A),\z)
\]
is studied extensively in \cite[\S3.2]{hutchinson2017} and \cite[\S3]{B-E--2023}. Most of the same techniques work 
for the study of the natural map 
\[
H_k(\PT(A),\z)\arr H_k(\PB(A),\z).
\] 
But using the following lemma we can transfer certain results from the subgroups $\Tt(A)\se \Bb(A)$
of $\SL_2(A)$ to the related subgroups $\PT(A)\se \PB(A)$ of $\PSL_2(A)$.

\begin{lem}\label{T-PT}
If $H_k(\Tt(A),\z)\!\simeq\! H_k(\Bb(A),\z)$ for $k\leq 3$, then 
$H_k(\PT(A),\z)\!\simeq\! H_k(\PB(A),\z)$ for $k\leq 3$.
\end{lem}
\begin{proof}
This follows from the Lyndon/Hochschild-Serre spectral sequence of the morphism of the
extensions
\[
\begin{tikzcd}
1 \ar[r]&\mu_2(A) \ar[r]\ar[d]&\Tt(A) \ar[r]\ar[d] &\PT(A)\ar[r]\ar[d]& 1\\
1 \ar[r]&\mu_2(A) \ar[r]&\Bb(A) \ar[r] &\PB(A)\ar[r]& 1.
\end{tikzcd}
\]
\end{proof}

\begin{lem}\label{wedge3}
Let $H_3(\PT(A),\z)\simeq H_3(\PB(A),\z)$ and $-1\in \aa^2$. Then the image of $(\bigwedge_\z^3 \PT(A))/2$ 
in $E_{0,3}^2$ sits in the image of $d^2_{2,2}$. In particular, $E_{0,3}^3$ is a quotient of the group
$\tors(\widetilde{\mu}(A),\widetilde{\mu}(A))^{\Sigma_2}$.
\end{lem}
\begin{proof}
As in the proof of Lemma \ref{surj1} we denote the matrix  $D(a)$ with $\hat{a}$. Let 
\[
\begin{array}{c}
\overline{\hat{a}\wedge \hat{b} \wedge \hat{c}}\in (\bigwedge_\z^3 \PT(A))/2
\end{array}
\]
and denote its image in $E_{0,3}^2$ by ${\bf c}(\hat{a}, \hat{b}, \hat{c})$.
It is sufficient to show that this element sits in the image of $d^2_{2,2}$. Let $i\in A$ 
such that $i^2=-1$ and consider the diagram
\[
\begin{tikzcd}
\hspace{-1cm}
B_3(\PSL_2(A))\otimes_{\PSL_2(A)} 
X_0(A^2) & B_3(\PSL_2(A))\otimes_{\PSL_2(A)} 
X_1(A^2) \ar["\id_{B_3}\otimes \partial_1"',  l] \ar[d, "d_3\otimes \id_{X_1}"] & \\
& B_2(\SL_2(A))\otimes_{\PSL_2(A)} 
X_0(A^2) & B_2(\PSL_2(A))\otimes_{\PSL_2(A)}  Z_1(A^2).\ar[l, "\id_{B_2}\otimes \inc"']
\end{tikzcd}
\]
Let $\lambda(a,b)=\overline{X_{a,b}}\in H_2(\PSL_2(A),Z_1(A^2))$, where the cycle $X_{a,b}$ is as in the 
proof of Lemma~\ref{surj1}.
Then 
\[
\Lambda(a,b,c):=\lambda(ab,c)-\lambda(a,c)-\lambda(b,c)=\overline{X_{a,b,c}}\in H_2(\PSL_2(A),Z_1(A^2)),
\]
where $X_{a,b,c}\in B_2(\PSL_2(A))\otimes_{\PSL_2(A)} Z_1(A^2)$ is the cycle
\begin{align*}
X_{a,b,c}:=&([\widehat{ab}|\hat{c}]-[\hat{a}|\hat{c}]-[\hat{b}|\hat{c}]+
[\hat{i}|\widehat{abc}]-[\hat{i}|\widehat{ab}]-[\hat{i}|\widehat{bc}]
-[\hat{i}|\widehat{ac}]+[\hat{i}|\hat{a}]+[\hat{i}|\hat{b}]+[\hat{i}|\hat{c}]\\
&+[{w}|\widehat{iabc}]-[{w}|\widehat{iac}]-[{w}|\widehat{ibc}]-[{w}|\widehat{iab}]+
[{w}|\widehat{ia}]+[{w}|\widehat{ib}]+[{w}|\widehat{ic}]-[{w}|\hat{i}])\otimes Y\\
&+[w\widehat{iabc}|w\widehat{iabc}]\otimes \partial_2(X_{abc})-
[w\widehat{iab}|w\widehat{iab}]\otimes \partial_2(X_{ab})
-[w\widehat{ibc}|w\widehat{ibc}]\otimes \partial_2(X_{bc})\\
&-[w\widehat{iac}|w\widehat{iac}]\otimes \partial_2(X_{ac})
+[w\widehat{ia}|w\widehat{ia}]\otimes \partial_2(X_a)
+[w\widehat{ib}|w\widehat{ib}]\otimes \partial_2(X_b)\\
&+[w\widehat{ic}|w\widehat{ic}]\otimes \partial_2(X_c) 
-[w\hat{i}|w\hat{i}]\otimes \partial_2(X_1).
\end{align*}
Let 
\begin{align*}
\theta_z&:=[{g^{-1}_z}|w\widehat{iz}|w\widehat{iz}]-[{h^{-1}_z}|w\widehat{iz}|w\widehat{iz}]
+[\hat{i}\hat{z}^{-1}|{g^{-1}_z}|w\widehat{iz}]-[\widehat{iz}|{h^{-1}_z}|w\widehat{iz}]\\
&+[\widehat{iz}|{\hat{i}\hat{z}^{-1}}|{g^{-1}_z}]-[\hat{i}\hat{z}^{-1}|\widehat{iz}|{h^{-1}_z}]
+[\hat{i}\hat{z}^{-1}|\widehat{iz}|\hat{i}\hat{z}^{-1}],
\end{align*}
where $g_z:=\overline{\mtxx{0}{1}{-1}{z}}$ and $h_z:=\overline{\mtxx{1}{z^{-1}}{0}{1}}$.
Let $Y_{a,b,c}\in B_3(\PSL_2(A))\otimes_{\PSL_2(A)} X_1(A^2)$ be defined as follows
\begin{align*}
Y_{a,b,c}:&=\{[{w}|\widehat{iab}|\hat{c}]-[{w}|\widehat{ia}|\hat{c}]
-[{w}|\widehat{ib}|\hat{c}]+[{w}|\hat{i}|\hat{c}]+[\hat{i}|\widehat{ab}|\hat{c}]
-[\hat{i}|\hat{a}|\hat{c}]-[\hat{i}|\hat{b}|\hat{c}]\}\otimes Y\\
&+(\Phi_{a,b, c}+\Psi_{a,b, c})\otimes ({\pmb \infty},{\pmb 0})),
\end{align*}
where
\begin{align*}
\Phi_{a,b, c}:=&\theta_{abc}-\theta_{ab}-\theta_{bc}-\theta_{ac}+\theta_{a}+\theta_{b}
+\theta_{c}-\theta_{1},\\
\Psi_{a,b, c}:=&[w\widehat{iab}|w\widehat{iab}|\hat{c}]-[w\widehat{ia}|w\widehat{ia}|\hat{c}]
-[w\widehat{ib}|w\widehat{ib}|\hat{c}]+[w\hat{i}|w\hat{i}|\hat{c}]\\
&+[\hat{c}|w\widehat{iabc}|w\widehat{iabc}]-[\hat{c}|w\widehat{iac}|w\widehat{iac}]
-[\hat{c}|w\widehat{ibc}|w\widehat{ibc}]+[\hat{c}|w\widehat{ic}|w\widehat{ic}]\\
&+[\widehat{iab}|w\widehat{iab}|\hat{c}]-[\widehat{ia}|w\widehat{ia}|\hat{c}]
-[\widehat{ib}|w\widehat{ib}|\hat{c}]+[\hat{i}|w\widehat{i}|\hat{c}]\\
&+[\widehat{iab}|\hat{c}|w\widehat{iabc}]-[\widehat{ia}|\hat{c}|w\widehat{iac}]
-[\widehat{ib}|\hat{c}|w\widehat{ibc}]+[\hat{i}|\hat{c}|w\widehat{ic}]\\
&-[\hat{c}|\widehat{iab}|w\widehat{iabc}]+[\hat{c}|\widehat{ia}|w\widehat{iac}]
+[\hat{c}|\widehat{ib}|w\widehat{ibc}]-[\hat{c}|\hat{i}|w\widehat{ic}]\\
&-[\hat{b}|\widehat{ia}|\hat{c}]+[\hat{b}|\hat{c}|\widehat{ia}]
-[\hat{c}|\hat{b}|\widehat{ia}]+[\hat{b}|\hat{i}|\hat{c}]
-[\hat{b}|\hat{c}|\hat{i}]+[\hat{c}|\hat{b}|\hat{i}].
\end{align*}
Then it is straightforward to check that
\[
(d_3\otimes \id_{X_1})(Y_{a,b,c})=(\id_{B_2}\otimes\inc)(X_{a,b,c}).
\]
Now 
\[
d_{2,2}^2(\Lambda(a,b,c))=\overline{(\id_{B_3}\otimes\partial_1)(Y_{a,b,c})},
\]
where
\[
(\id_{B_3}\otimes\partial_1)(Y_{a,b,c})\in B_3(\PSL_2(A))\otimes_{\PSL_2(A)} X_0(A^2).
\]
Since $\partial_1(Y)=0$, we need only to study
\[
(\id_{B_3}\otimes \partial_1)((\Phi_{a,b, c}+\Psi_{a,b, c})\otimes ({\pmb \infty},{\pmb 0})).
\]
We have
\[
(\id_{B_3}\otimes\partial_1)(Y_{a,b,c})=
(w\Phi_{a,b, c}-\Phi_{a,b, c})\otimes ({\pmb \infty})+ (w\Psi_{a,b, c}-\Psi_{a,b, c})\otimes ({\pmb \infty}).
\]
Now consider the composite
\[
B_3(\PSL_2(A))\otimes_{\PSL_2(A)} X_0(A^2) \arr C_3(\PSL_2(A))\otimes_{\PSL_2(A)} X_0(A^2) 
\arr C_3(\PSL_2(A))\otimes_{\PB(A)} \z.
\]
Then
\begin{align*}
(w\theta_z-\theta_z)\otimes ({\pmb \infty})\!\mt&\Big\{({w},{wg^{-1}_z},{wg^{-1}_zw\widehat{iz}},{wg^{-1}_z})
-({w},{wh^{-1}_z},{wh^{-1}_zw\widehat{iz}},{wh^{-1}_z})\\
&+({w},{w\hat{i}\hat{z}^{-1}},{w\hat{i}\hat{z}^{-1}g^{-1}_z},{w\hat{i}\hat{z}^{-1}g^{-1}_zw\widehat{iz}})
-({w},{w\widehat{iz}},{w\widehat{iz}h^{-1}_z},{w\widehat{iz}h^{-1}_zw\widehat{iz}})\\
&+({w},{w\widehat{iz}},{w},{wg^{-1}_z})-({w},{w\hat{i}\hat{z}^{-1}},{w},{wh^{-1}_z})
-({w},{w\hat{i}\hat{z}^{-1}},{w},{w\hat{i}\hat{z}^{-1}})\\
&-(1,{g^{-1}_z},{g^{-1}_zw\widehat{iz}},{g^{-1}_z})+(1,{h^{-1}_z},{h^{-1}_zw\widehat{iz}},{h^{-1}_z})\\
&-(1,{\hat{i}\hat{z}^{-1}},{\hat{i}\hat{z}^{-1}g^{-1}_z},{\hat{i}\hat{z}^{-1}g^{-1}_zw\widehat{iz}})
+(1,\widehat{iz},{\widehat{iz}h^{-1}_z},{\widehat{iz}h^{-1}_zw\widehat{iz}})\\
&-(1,\widehat{iz},1,{g^{-1}_z})+(1,{\hat{i}\hat{z}^{-1}},1,{h^{-1}_z})
+(1,{\hat{i}\hat{z}^{-1}},1,{\hat{i}\hat{z}^{-1}})\Big\}\otimes 1
\end{align*}
and
\begin{align*}
(w\Psi_{a,b, c}\!-\!\Psi_{a,b, c})\otimes\! ({\pmb \infty})\!\mt
&\{({w},\widehat{iab},{w},w\hat{c})\!-\!({w},\widehat{ia},{w},w\hat{c})\!
-\!({w},\widehat{ib},{w},w\hat{c})+({w},\hat{i},{w},w\hat{c})\\
&+\!({w},w\hat{c},\widehat{iab},w\hat{c})\!-\!({w},w\hat{c},\widehat{ia},w\hat{c})\!
-\!({w},w\hat{c},\widehat{ib},w\hat{c})\!+\!({w},w\hat{c},\hat{i},w\hat{c})\\
&+({w},w\widehat{iab},1,\hat{c})-({w},w\widehat{ia},1,\hat{c})
-({w},w\widehat{ib},1,\hat{c})+({w},w\hat{i},1,\hat{c})\\
&+({w},w\widehat{iab},w\widehat{iabc},1)-({w},w\widehat{ia},w\widehat{iac},1)
-({w},w\widehat{ib},w\widehat{ibc},1)\\
&+({w},w\hat{i},w\widehat{ic},1)-({w},w\hat{c},w\widehat{iabc},1)+({w},w\hat{c},w\widehat{iac},1)\\
&+({w},w\hat{c},w\widehat{ibc},1)-({w},w\hat{c},w\widehat{ic},1)
-({w},w\hat{b},w\widehat{iab},w\widehat{iabc})\\
&+({w},w\hat{b},w\widehat{bc},w\widehat{iabc})-({w},w\hat{c},w\widehat{bc},w\widehat{iabc})\\
&+({w},w\hat{b},w\widehat{ib},w\widehat{ibc})-({w},w\hat{b},w\widehat{bc},w\widehat{ibc})
+({w},w\hat{c},w\widehat{bc},w\widehat{ibc})\\
&-(1,w\widehat{iab},1,\hat{c})+(1,w\widehat{ia},1,\hat{c})
+(1,w\widehat{ib},1,\hat{c})-(1,w\hat{i},1,\hat{c})\\
&-(1,\hat{c},w\widehat{iab},\hat{c})+(1,\hat{c},w\widehat{ia},\hat{c})
+(1,\hat{c},w\widehat{ib},\hat{c})-(1,\hat{c},w\hat{i},\hat{c})\\
&-(1,\widehat{iab},{w},w\hat{c})+(1,\widehat{ia},{w},w\hat{c})
+(1,\widehat{ib},{w},w\hat{c})-(1,\hat{i},{w},w\hat{c})\\
&-(1,\widehat{iab},\widehat{iabc},{w})+(1,\widehat{ia},\widehat{iac},{w})
+(1,\widehat{ib},\widehat{ibc},{w})-(1,\hat{i},\widehat{ic},{w})\\
&+(1,\hat{c},\widehat{iabc},{w})-(1,\hat{c},\widehat{iac},{w})
-(1,\hat{c},\widehat{ibc},{w})+(1,\hat{c},\widehat{ic},{w})\\
&+(1,\hat{b},\widehat{iab},\widehat{iabc})-(1,\hat{b},\widehat{bc},\widehat{iabc})
+(1,\hat{c},\widehat{bc},\widehat{iabc})\\
&-(1,\hat{b},\widehat{ib},\widehat{ibc})+(1,\hat{b},\widehat{bc},\widehat{ibc})
-(1,\hat{c},\widehat{bc},\widehat{ibc})\}\otimes 1.
\end{align*}
Now we want to follow these elements through  the maps
\[
C_3(\PSL_2(A))\otimes_{\PB(A)} \z \overset{s_3}{\larr}  C_3(\PB(A))\otimes_{\PB(A)} 
\z \arr C_3(\PT(A))\otimes_{\PT(A)} \z 
\]
\[
\arr B_3(\PT(A))\otimes_{\PT(A)} \z,
\]
where $s_3$ is induced by the section 
\[
s:\PT(A)\backslash\PSL_2(A)\to \PSL_2(A)
\]
discussed in the proof of Lemma \ref{surj1} (see also Lemma \ref{G-H}).
It is straightforward to check that modulo $\im(d_4)$ we have	
\[
(w\theta_z-\theta_z)\otimes({\pmb \infty})\mapsto
(-[\hat{i}\hat{z}^{-1}|\widehat{iz}|\hat{i}\hat{z}^{-1}]+[\hat{i}|\hat{i}|\hat{i}])\otimes 1
=([\widehat{iz}|\hat{i}\hat{z}^{-1}|\widehat{iz}]-[\hat{i}|\hat{i}|\hat{i}])\otimes 1.
\]	
Moreover,
\begin{align*}
(w\Phi_{a,b,c}-\Phi_{a,b,c})\otimes ({\pmb \infty})\mt&
\bigg\{[\widehat{iabc}|(\widehat{iabc})^{-1}|\widehat{iabc}]-[\widehat{iab}|(\widehat{iab})^{-1}|\widehat{iab}]
-[\widehat{ibc}|(\widehat{ibc})^{-1}|\widehat{ibc}]\\
&-[\widehat{iab}|(\widehat{iac})^{-1}|\widehat{iac}]+[\widehat{ia}|(\widehat{ia})^{-1}|\widehat{ia}]
+[\widehat{ib}|(\widehat{ib})^{-1}|\widehat{ib}]\\
&+[\widehat{ic}|(\widehat{ic})^{-1}|\widehat{ic}]-[\hat{i}|\hat{i}|\hat{i}]\bigg\}\otimes 1,
\end{align*}
\begin{align*}
(w\Psi_{a,b,c}-\Psi_{a,b,c})\otimes ({\pmb \infty})\mt&
\bigg\{[\hat{c}^{-1}|\widehat{iabc}|(\widehat{iabc})^{-1}]
-[\hat{c}^{-1}|\widehat{iac}|(\widehat{iac})^{-1}]-[\hat{c}^{-1}|\widehat{ibc}|(\widehat{ibc})^{-1}]\\
&+[\hat{c}^{-1}|\widehat{ic}|(\widehat{ic})^{-1}]+[(\widehat{iab})^{-1}|\hat{c}^{-1}|\widehat{iabc}]
-[(\widehat{ia})^{-1}|\hat{c}^{-1}|\widehat{iac}]\\
&-[(\widehat{ib})^{-1}|\hat{c}^{-1}|\widehat{ibc}]+[\hat{i}|\hat{c}^{-1}|\widehat{ic}]
-[\hat{c}^{-1}|(\widehat{iab})^{-1}|\widehat{iabc}]+[\hat{c}^{-1}|(\widehat{ia})^{-1}|\widehat{iac}]\\
&+[\hat{c}^{-1}|(\widehat{ib})^{-1}|\widehat{ibc}]-[\hat{c}^{-1}|\hat{i}|\widehat{ic}]
-[\hat{b}^{-1}|(\widehat{ia})^{-1}|\hat{c}^{-1}]+[\hat{b}^{-1}|\hat{c}^{-1}|(\widehat{ia})^{-1}]\\
&-[\hat{c}^{-1}|\hat{b}^{-1}|(\widehat{ia})^{-1}]+[\hat{b}^{-1}|\hat{i}|\hat{c}^{-1}]
-[\hat{b}^{-1}|\hat{c}^{-1}|\hat{i}]+[\hat{c}^{-1}|\hat{b}^{-1}|\hat{i}]\\
&-[\hat{c}|(\widehat{iabc})^{-1}|\widehat{iabc}]+[\hat{c}|(\widehat{iac})^{-1}|\widehat{iac}]
+[\hat{c}|(\widehat{ibc})^{-1}|\widehat{ibc}]-[\hat{c}|(\widehat{ic})^{-1}|\widehat{ic}]\\
&-[\widehat{iab}|\hat{c}|(\widehat{iabc})^{-1}]+[\widehat{ia}|\hat{c}|(\widehat{iac})^{-1}]
+[\widehat{ib}|\hat{c}|(\widehat{ibc})^{-1}]-[\hat{i}|\hat{c}|(\widehat{ic})^{-1}]\\
&+[\hat{c}|\widehat{iab}|(\widehat{iabc})^{-1}]-[\hat{c}|\widehat{ia}|(\widehat{iac})^{-1}]
-[\hat{c}|\widehat{ib}|(\widehat{ibc})^{-1}]+[\hat{c}|\hat{i}|(\widehat{ic})^{-1}]\\
&+[\hat{b}|\widehat{ia}|\hat{c}]-[\hat{b}|\hat{c}|\widehat{ia}]+[\hat{c}|\hat{b}|\widehat{ia}]
-[\hat{b}|\hat{i}|\hat{c}]+[\hat{b}|\hat{c}|\hat{i}]-[\hat{c}|\hat{b}|\hat{i}]\bigg\}\otimes 1.
\end{align*}
Combining all these and adding the null element
\begin{align*}
d_4(\bigg\{&-[\hat{c}|\hat{c}^{-1}|\widehat{iabc}|(\widehat{iabc})^{-1}]
+[\hat{c}|\hat{c}^{-1}|\widehat{ia}c|(\widehat{iac})^{-1}]
+[\hat{c}|\hat{c}^{-1}|\widehat{ibc}|(\widehat{ibc})^{-1}]\\
&-[\hat{c}|\hat{c}^{-1}|\widehat{ic}|(\widehat{ic})^{-1}]
+[\hat{c}|\hat{c}^{-1}|(\widehat{iab})^{-1}|\widehat{iabc}]
-[\hat{c}|\hat{c}^{-1}|(\widehat{ia})^{-1}|\widehat{iac}]
-[\hat{c}|\hat{c}^{-1}|(\widehat{ib})^{-1}|\widehat{ibc}]\\
&+[\hat{c}|\hat{c}^{-1}|\hat{i}|\widehat{ic}]
-[\widehat{iabc}|(\widehat{iab})^{-1}|\hat{c}^{-1}|\widehat{iabc}]
+[\widehat{iac}|(\widehat{ia})^{-1}|\hat{c}^{-1}|\widehat{iac}]
+[\widehat{ibc}|(\widehat{ib})^{-1}|\hat{c}^{-1}|\widehat{ibc}]\\
&-[\widehat{ic}|\hat{i}|\hat{c}^{-1}|\widehat{ic}]
+[\hat{c}|\widehat{iab}|(\widehat{iab})^{-1}|\widehat{iab}]
-[\hat{c}|\widehat{ia}|(\widehat{ia})^{-1}|\widehat{ia}]
-[\hat{c}|\widehat{ib}|(\widehat{ib})^{-1}|\widehat{ib}]+[\hat{c}|\hat{i}|\hat{i}|\hat{i}]\\
&+[\widehat{iabc}|\hat{b}^{-1}|(\widehat{ia})^{-1}|\hat{c}^{-1}]
-[\widehat{iabc}|\hat{b}^{-1}|\hat{c}^{-1}|(\widehat{ia})^{-1}]
-[\hat{c}|\widehat{ia}|\hat{b}|(\widehat{iab})^{-1}]-[\widehat{ia}|\hat{b}|\hat{c}|(\widehat{iabc})^{-1}]\\
&+[\widehat{ia}|\hat{c}|\hat{c}^{-1}|(\widehat{ia})^{-1}]
-[\widehat{ia}|\widehat{bc}|\hat{b}^{-1}|\hat{c}^{-1}]
+[\widehat{ia}|\hat{c}|\hat{b}|\hat{b}^{-1}]+[\widehat{iac}|\hat{b}|\hat{b}^{-1}|(\widehat{ia})^{-1}]\\
&-[\widehat{ia}|\widehat{bc}|(\widehat{bc})^{-1}|(\widehat{ia})^{-1}]
-[\hat{b}|\hat{c}|(\widehat{bc})^{-1}|(\widehat{ia})^{-1}]
+[\hat{c}|\hat{c}^{-1}|\hat{b}^{-1}|(\widehat{ia})^{-1}]+[\hat{i}|\hat{b}|\hat{c}|(\widehat{ibc})^{-1}]\\
&-[\hat{i}|\hat{c}|\hat{b}|(\widehat{ibc})^{-1}]+[\widehat{ic}|\hat{b}|(\widehat{ib})^{-1}|\hat{c}^{-1}]
+[\hat{c}|\hat{i}|\hat{b}|(\widehat{ib})^{-1}]+[\hat{b}|\hat{c}|(\widehat{bc})^{-1}|\hat{i}]
-[\hat{c}|\hat{b}|(\widehat{bc})^{-1}|\hat{i}]\\
&-[\hat{c}|\hat{c}^{-1}|\hat{b}^{-1}|\hat{i}]+[\hat{b}|\hat{b}^{-1}|\hat{c}^{-1}|\hat{i}]
-[\hat{b}|\hat{b}^{-1}|\hat{i}|\hat{c}^{-1}]
-[\hat{c}|\hat{b}|\hat{b}^{-1}|\hat{c}^{-1}]\bigg\}\otimes 1)
\end{align*}
we see that 
$d_{2,2}^2(\Lambda(a,b,c))=\overline{Z_{a,b,c}}$, where $Z_{a,b,c}$ is the cycle
\begin{align*}
Z_{a,b,c}= &-([\widehat{ia}|\hat{b}|\hat{c}]+[\hat{c}|\widehat{ia}|\hat{b}]+[\hat{b}|\hat{c}|\widehat{ia}]
-[\hat{b}|\widehat{ia}|\hat{c}]-[\hat{c}|\hat{b}|\widehat{ia}]-[\widehat{ia}|\hat{c}|\hat{b}]\\
&-[\hat{i}|\hat{b}|\hat{c}]-[\hat{c}|\hat{i}|\hat{b}]-[\hat{b}|\hat{c}|\hat{i}]+[\hat{b}|\hat{i}|\hat{c}]
+[\hat{c}|\hat{b}|\hat{i}]+[\hat{i}|\hat{c}|\hat{b}])\otimes 1.
\end{align*}
Therefore, 
\begin{align*}
d_{2,2}^2(\Lambda(a,b,c)) =\overline{Z_{a,b,c}}=-\Big({\bf c}(\widehat{ia}, \hat{b}, \hat{c})-{\bf c}(\hat{i}, \hat{b}, \hat{c})\Big)
=-{\bf c}(\hat{a}, \hat{b}, \hat{c}).
\end{align*}
This completes the proof of the lemma.
\end{proof}

\begin{cor}
Let $H_3(\PB(A),\PT(A);\z)=0$ and $-1\in \aa^2$. Then the composite
\[
\begin{array}{c}
\bigwedge_\z^3 \PT(A) \arr H_3(\PT(A), \z)\arr H_3(\PSL_2(A), \z)
\end{array}
\]
is trivial.
\end{cor}
\begin{proof}
This follows immediately from the previous lemma.
\end{proof}

\section{A projective refined Bloch-Wigner exact sequence}\label{sec3}

Here is the main result of this paper.

\begin{thm}\label{Proj-BW}
Let $A$ be a $\GE_2$-ring such that $H_3(\PT(A),\z)\simeq H_3(\PB(A),\z)$.
If moreover $-1\in \aa^2$, then we have the projective refined Bloch-Wigner exact sequence
\[
\tors(\widetilde{\mu}(A),\widetilde{\mu}(A))^{\Sigma_2'} \arr H_3(\PSL_2(A),\z) \arr \RB'(A) \arr 0.
\]
If $A$ is a domain, then we have the exact sequence
\[
0\arr \tors(\widetilde{\mu}(A),\widetilde{\mu}(A)) \arr H_3(\PSL_2(A),\z) \arr \RB'(A) \arr 0.
\]
\end{thm}
\begin{proof}
We have seen that $E_{2,1}^2\simeq \RP_1'(A)$ and $E_{1,1}^2\simeq \widetilde{\mu}_4(A)$ (see Lemma \ref{d21}). 
On the other hand, $E_{0,2}^2\simeq H_2(\PB(A),\z)$ (see the paragraph above Lemma~\ref{d21}) and by 
Lemma~\ref{surj1}, $E_{1,2}^2=0$. Since $H_3(\PT(A),\z)\simeq H_3(\PB(A),\z)$, we have $E_{0,3}^2\simeq \LL$
(see the paragraph following Proposition \ref{H3A}). Also note that $E_{0,1}^2\simeq \widetilde{\GG}_A\oplus A_{\aa}$
(see the paragraph above Lemma \ref{GG}) and $E_{02}^2\simeq I'(A)$ (see Remark \ref{IA}).
Finally by Lemma \ref{coincide1}, the differential 
\[
d_{2,1}^2:\RP_1'(A) \arr H_2(\PB(A),\z)
\]
coincides with $\lambda_1'$.  Therefore the second page $E_{p,q}^2$ of the spectral sequence has the following form:
\[
\begin{tikzcd}
\LL & \ast & \ast & 0  \\
H_2(\PB(A),\z) & 0 &  \ar[llu,"d^2_{2,2}"']E_{2,2}^2 & 0  \\
\widetilde{\GG}_A\oplus A_{\aa} & \widetilde{\mu}_4(A) & \ar[llu,"d^2_{2,1}=\lambda_1'"'] \RP_1'(A) & 0  \\
\z & 0 & \ar[llu,"d^2_{2,0}"'] I'(A) & 0 
\end{tikzcd}
\]
Thus by Lemma \ref{coincide1}, $E_{2,1}^\infty\simeq E_{2,1}^3\simeq \RB'(A)$ and
by Lemma \ref{wedge3}, $E_{0,3}^3$ is a quotient of $\tors(\tilde{\mu}(A),\tilde{\mu}(A))^{\Sigma_2'}$.
Now by an easy analysis of the spectral sequence we obtain the exact sequence
\[
\tors(\widetilde{\mu}(A),\widetilde{\mu}(A))^{\Sigma_2'} \arr H_3(\PSL_2(A),\z) \arr 
\RB'(A) \arr 0.
\]
Now let $A$ be a domain. Then 
$\tors(\widetilde{\mu}(A),\widetilde{\mu}(A))^{\Sigma_2'} =\tors(\widetilde{\mu}(A),
\widetilde{\mu}(A))$ (see the paragraph following Proposition \ref{H3A}).
Let $F$ be the quotient field of $A$ and $\overline{F}$ the algebraic 
closure of $F$. It is very easy to see that 
\[
\RB'(\overline{F})=\RB(\overline{F})\simeq\BB(\overline{F}),
\]
where $\BB(\overline{F})$ is the classical Bloch group of $\overline{F}$ \cite[\S1]{suslin1991},
\cite[\S2.2]{hutchinson2017}. 
The classical Bloch-Wigner exact sequence claims that the sequence
\[
0\arr\tors(\widetilde{\mu}(\overline{F}),\widetilde{\mu}(\overline{F}))\arr H_3(\PSL_2(\overline{F}),\z) 
\arr \BB(\overline{F})\arr 0
\]
is exact (see \cite[Theorem, App. A]{dupont-sah1982} or \cite[Theorem 5.2]{suslin1991}). Now the final 
claim follows from the commutative diagram with exact rows
\[
\begin{tikzcd}
&\tors(\widetilde{\mu}(A),\widetilde{\mu}(A))\ar[r]\ar[d, hook]&H_3(\PSL_2(A),\z) \ar[r]\ar[d]
&\RB'(A)\ar[r]\ar[d]&0\\
0 \ar[r] & \tors (\widetilde{\mu}(\overline{F}), \widetilde{\mu}(\overline{F})) \ar[r] & 
H_3(\PSL_2(\overline{F}),\z) \ar[r] & \BB(\overline{F}) \ar[r] & 0
\end{tikzcd}
\]
and the fact that the left vertical map is injective.
\end{proof}

\begin{thm}\label{G<4}
Let $A$ be a $\GE_2$-domain such that  
$H_3(\PT(A),\z)\simeq H_3(\PB(A),\z)$. If $|\GG_A|\leq 4$, then we have the exact sequence
\[
0 \arr \tors(\widetilde{\mu}(A),\widetilde{\mu}(A)) \arr H_3(\PSL_2(A), \z)\arr \RB'(A) \arr 0.
\]
\end{thm}
\begin{proof}
If $-1\in \aa^2$, then the claim follows from the previous theorem. 
So let $-1\notin \aa^2$. Then  $\GG_A=\{1, -1, a, -a\}$ for some $a\in \aa$. 
It follows from this that
\[
\GG_A \wedge \GG_A= \GG_A \wedge \mu_2(A)
\]
and hence by Lemma~\ref{GG}, $E_{1,2}^2=0$. Moreover, since $\GG_A$ is isomorphic with a subgroup of 
$\z/2 \times \z/2$, we have $\bigwedge_\z^3 \GG_A=0$. Therefore 
\[
E_{0,3}^2\simeq \tors(\widetilde{\mu}(A),\widetilde{\mu}(A))
\]
(see Lemma \ref{wedge} and the paragraph following Proposition \ref{H3A}). 
Now the claim follows from an easy analysis of the main spectral sequence. 
\end{proof}

\begin{cor}\label{PRBW-cor}
Let $A$ be a local domain such that its residue field either is infinite or if it has 
$p^d$ elements, then $(p-1)d>6$. If moreover either $-1\in \aa^2$ or $|\GG_A|\leq 4$, then we 
have the exact sequence
\[
0\arr \tors(\widetilde{\mu}(A),\widetilde{\mu}(A)) \arr H_3(\PSL_2(A),\z) \arr \RB(A) \arr 0.
\]
\end{cor}
\begin{proof}
By Example \ref{exact-11}, $X_\bullet(A^2)\arr \z$ is exact in dimension $<2$. Thus
\[
\RB(A)=\RB'(A)
\]
(see Remark \ref{RPRP'}). Moreover, by \cite[Proposition 3.19]{hutchinson2017} and  Lemma \ref{T-PT}, 
\[
H_3(\PT(A),\z)\simeq H_3(\PB(A),\z).
\]
Now the claim follows from the previous two theorems.
\end{proof}

\begin{cor}\label{DVR}
Let $A$ be a discrete valuation ring with quotient field $F$ and residue field $k$. If 
$|k|=p^d$, we assume that $(p-1)d>6$. If $|\GG_F|\leq 8$, then we have the exact sequence
\[
0\arr \tors(\widetilde{\mu}(A),\widetilde{\mu}(A)) \arr H_3(\PSL_2(A),\z)\arr \RB(A) \arr 0.
\]
\end{cor}

\begin{proof}
From the valuation of $F$ we obtain the exact sequence 
\[
1 \arr \GG_{A} \arr \GG_{F} \arr \z/2 \arr 0.
\]
It follows from this that $|\GG_A|\leq 4$ and hence the claim follows from the above corollary.
\end{proof}

\begin{cor}\label{fields}
Let $F$  be either a  quadratically closed field, a real closed field or a local field which is 
not a finite extension of $\q_2$. Then we have the exact sequence
\[
0 \arr \tors(\widetilde{\mu}(F),\widetilde{\mu}(F)) \arr H_3(\PSL_2(F), \z)\arr \RB(F) \arr 0.
\]
\end{cor}
\begin{proof}
If $F$ is  quadratically closed, then clearly 
$\GG_F=1$. If $F$ is real closed, then $F^\times\simeq \mu_2(F) \oplus F^{>0}$ and 
$F^{>0}=(F^{>0})^2$. This implies that $\GG_F=\{1,-1\}\simeq \z/2$. Now let $F$ be a 
local field which is not a finite extension of $\q_2$. If $\char(F)=2$, then 
$-1=1\in {F^\times}^2$. If $\char(F)\neq 2$, then the residue field of its valuation ring is of 
odd characteristic. In this case $\GG_F$ has four elements
\cite[Theorem 2.2, Chap. VI]{lam2005}. Therefore the claim follows from Corollary \ref{PRBW-cor}.
\end{proof}

\begin{cor}\label{finite}
Let $\F_q$ be a finite field with $q=p^d$ elements. If $(p-1)d>6$, then we have the exact sequence
\[
0\arr \tors(\widetilde{\mu}(\F_q),\widetilde{\mu}(\F_q)) \arr H_3(\PSL_2(\F_q), \z)\arr \RB(\F_q) \arr 0.
\]
\end{cor}
\begin{proof}
Since $\GG_{\F_q}$ has two elements, the claim follows from Corollary \ref{PRBW-cor}.
\end{proof}

\begin{cor}
Let $p$ be either $1$ or a prime number. Then 
\[
\begin{array}{c}
H_3(\PSL_2(\z\pth),\z)\simeq \RB'(\z\pth).
\end{array}
\]
In particular,
for $p=1, 2$ or $3$ we have  $H_3(\PSL_2(\z\pth),\z)\simeq \RB(\z\pth)$.
\end{cor}
\begin{proof}
It is known that  $A:=\z\pth$ is a $\GE_2$-ring. In fact any euclidean domain is a $\GE_2$-ring 
\cite[\S2]{cohn1966}.
Note that $\z^\times= \mu_2(\z)$ and for a prime $p$,
\[
\begin{array}{c}
\aa=(\z\pth)^\times= \{\pm 1\}\times \lan p\ran\simeq \z/2\times \z.
\end{array}
\]
Thus $\GG_A$ has at most $4$ elements and $\widetilde{\mu}(A)=1$. By Theorem \ref{G<4}, to complete the 
proof, we need to prove that $H_3(\PT(A),\z)\simeq H_3(\PB(A),\z)$.
Since $\PB(\z)=\PN(\z)\simeq\z$, we have 
\[
H_3(\PB(\z),\z)=0=H_3(\PT(\z),\z).
\]
Let $p$ be a prime. With an argument similar to the proof of \cite[Lemma 3.5]{B-E--2023}
we obtain the isomorphism
\[
H_n(\PB(A),\z)\simeq H_n(\PT(A),\z)\oplus H_{n-1}(\PT(A), A),
\]
where $\PT(A)$ acts on $A$ by the formula $D(a).x=a^2x$.
Since $\PT(A)\simeq \z$, we have $H_{m}(\PT(A), A)=0$ for $m\geq 2$. 
Thus 
\[
H_n(\PB(A),\z)\simeq H_n(\PT(A),\z), \ \ \ \text{for $n\geq 3$.}
\]
This implies the first claim. The second claim follows 
from the fact that $X_\bullet(A^2)\arr \z$ is exact in dimension $<2$ for $A=\z$, $\z\half$ and 
$\z[\frac{1}{3}]$ \cite[Examples 6.12, 6.13]{hutchinson2022}.
\end{proof}

\begin{rem}\label{RB(Z)}
(i) There is a bit difference between the refined Bloch group of $\z$ defined in here and in 
\cite{C-H2022}. Let denote the refined Bloch group of $\z$ define in \cite[page 16]{C-H2022} by
$\RB_{\rm ch}(\z)$. Then it was shown that we have the exact sequence
\[
0\arr \z/4 \arr H_3(\SL_2(\z),\z) \arr \RB_{\rm ch}(\z) \arr 0
\]
(see \cite[Theorem 8.16]{C-H2022}), where the left injective map is induce by the 
inclusion $\lan w\ran \harr \SL_2(\z)$. On the other hand, since $H_3(\SL_2(\z),\z)\simeq \z/12$, 
$H_3(\PSL_2(\z),\z)\simeq \z/6$ and the composition
\[
\z/4\simeq H_3(\lan w\ran,\z) \arr H_3(\SL_2(\z),\z) \arr H_3(\PSL_2(\z),\z)
\]
is trivial, we obtain the exact sequence
\[
0\arr \z/4 \arr H_3(\SL_2(\z),\z) \arr H_3(\PSL_2(\z),\z) \arr \z/2 \arr 0.
\]
Now from these and  the above corollary we obtain the exact sequence
\[
0\arr \RB_{\rm ch}(\z) \arr \RB(\z) \arr \z/2 \arr 0.
\]
\par (ii) Since $H_2(\Tt(\z\half),\z)\simeq H_2(\Bb(\z\half),\z)$ \cite[Lemma 8.28]{C-H2022} and 
$H_2(\PT(\z\half),\z)\simeq H_2(\PB(\z\half),\z)$, we have 
\[
\begin{array}{c}
\RB_{\rm ch}(\z\half)\simeq \RB(\z\half)
\end{array}
\]
(see Remark \ref{IA}). Thus by \cite[Proposition 8.31]{C-H2022}, we have the exact sequence
\[
\begin{array}{c}
0\arr \z/4 \arr H_3(\SL_2(\z\half),\z) \arr \RB(\z\half) \arr 0.
\end{array}
\]
The structure of $\RB_{\rm ch}(\z\half)$ has been studied extensively by Coronado and Hutchinson in 
\cite[\S8.4]{C-H2022}.
\par (iii) By a result of \cite[Lemma 4.6]{C-H2022}, there is a natural surjective map 
\[
\begin{array}{c}
H_3(\SL_2(\z\third),\z) \two \RB_{\rm ch}(\z\third).
\end{array}
\]
But the natural map 
\[
\begin{array}{c}
H_3(\SL_2(\z\third),\z) \arr H_3(\PSL_2(\z\third),\z)\simeq \RB(\z\third)
\end{array}
\]
is not surjective. In fact, we have the exact sequence
\[
\begin{array}{c}
H_3(\SL_2(\z\third),\z) \arr H_3(\PSL_2(\z\third),\z) \arr \z/2\arr 0
\end{array}
\]
(see \cite[Proposition 4.1]{B-E-2025}). Thus $\RB_{\rm ch}(\z\third)$ and $\RB(\z\third)$ are not isomorphic.
\end{rem}

\begin{cor}\label{Od}
Let $\OO_d$ be the ring of algebraic integers of the number field $\q(\sqrt{d})$, where 
$d\in \{-1, -2, -3, -7, -11\}$. Then we have the exact sequence
\[
\begin{array}{c}
0\arr\tors(\widetilde{\mu}(\OO_d),\widetilde{\mu}(\OO_d))\arr H_3(\PSL_2(\OO_d),\z)\arr 
\RB'(\OO_d) \arr 0.
\end{array}
\]
In particular, if $d=-1,-3$, then we may replace $\RB'(\OO_d)$ with $\RB(\OO_d)$.
\end{cor}
\begin{proof}
It is known that $\OO_{-1}^\times$ has $4$ elements, $\OO_{-3}^\times$ has $6$ elements and 
$\OO_{d}^\times=\{1,-1\}$ otherwise. Thus $|\GG_{\OO_d}|\leq 2$. Moreover, if
$d\neq-1,-3$, then $\PB(\OO_d)\simeq \z$ and thus 
\[
H_3(\PB(\OO_d),\z)=0=H_3(\PT(\OO_d),\z). 
\]
For $d=-1,-3$, we have
\[
\PB(\OO_{-1})\simeq \z \times \z/2, \ \ \ \  \PT(\OO_{-1})\simeq \z/2, 
\]
\[
\PB(\OO_{-3})\simeq \z \times \z/3, \ \ \ \ \PT(\OO_{-3})\simeq \z/3.
\]
Therefore 
\[
H_3(\PB(\OO_d),\z)\simeq H_3(\PT(\OO_d),\z).
\]
Now the exact sequence follows from Theorem~\ref{G<4}. For the last claim, note that for 
$A=\OO_{-1}$ and $A=\OO_{-3}$, $H_1(X_\bullet(A^2))=0$ (see \cite[Proposition~1.1]{B-E--2023} 
and \cite[Theorem 5.2, \S6]{cohn1966}). It follows from this that $\RP'(A)=\RP(A)$ (Remark 
\ref{RPRP'}) an thus $\RB'(A)=\RB(A)$.
\end{proof}

\begin{exa}
For $d=-1.-2,-3,-7,-11$, the homology groups of $\OO_d$ have been calculated by Schwermer and Vogtmann 
in \cite{S-V1983}. In particular, they showed that
\[
H_3(\PSL_2(\OO_d),\z)\simeq 
\begin{cases}
\z/6\times (\z/2)^3  & \text{if $d=-1$} \\
\z/6\times \z/2      & \text{if $d=-2$} \\
\z/6 \times \z/3     & \text{if $d=-3$}\\
\z/6 & \text{if $d=-7,-11$}
\end{cases}
\]
(see \cite[Theorems 5.5, 5.3, 5.7, 5.9, 5.11]{S-V1983}). Thus by Corollary \ref{Od},
\[
\RB'(\OO_d)\simeq 
\begin{cases}
\z/6\times (\z/2)^2  & \text{if $d=-1$} \\
\z/6\times \z/2      & \text{if $d=-2$} \\
\z/6     & \text{if $d=-3,-7,-11$.}
\end{cases}
\]
\end{exa}

\begin{exa}
Let $A=\OO_d$ be the ring of algebraic integers of $\q(\sqrt{d})$, where $d$ is a positive square free integer.
Then $A^\times=\{\pm u^n\}\simeq\z/2\times\z$, where
$u$ is a fundamental units of $A$. Therefore $|\GG_A|=4$ and
\[
\PT(A)\simeq \z.
\]
Since $\aa$ is infinite, $A$ is a $\GE_2$-ring \cite[page 321, Theorem]{vas1972}. Let study $H_3(\PB(A),\z)$.
Consider the extension 
\[
1 \arr \PN(A) \arr \PB(A) \arr \PT(A)\arr 1.
\]
Since $\PN(A)\simeq \z\times \z$, $H_n(\PN(A), \z)=0$ for $n\geq 3$ and 
\[
H_2(\PN(A), \z) \simeq \PN(A)\wedge \PN(A) \simeq \z.
\]
Now easy analysis of the Lyndon/Hochschild-Serre spectral sequence of the above extension implies that 
\[
H_3(\PB(A),\z)\simeq H_1(\PT(A), H_2(\PN(A),\z))=(\PN(A)\wedge \PN(A))^{\PT(A)}\simeq (A\wedge A)^{\PT(A)}
\]
(see \cite[Example 6.1.4]{weibel1994}). Let $x= a+b\sqrt{d}$ and  $y=a'+b'\sqrt{d}$. Then
\[
x\wedge y=(a+b\sqrt{d})\wedge (a'+b'\sqrt{d})=(ab'-a'b)(1\wedge \sqrt{d}).
\]
Now if $u=r+s\sqrt{d}\in \aa$, then $u^2=(r^2+s^2d)+2rs\sqrt{d}=r'+s'\sqrt{d}$. Thus we have 
\[
D(u).(x\wedge y)=(ab'-a'b)(u^2\wedge u^2\sqrt{d})=(ab'-a'b)((r'+s'\sqrt{d}) \wedge (s'd+r'\sqrt{d}))
\]
\[
=(ab'-a'b)(r'^2-s'^2d)(1\wedge \sqrt{d})=(ab'-a'b)(1\wedge \sqrt{d})=x\wedge y.
\]
This shows that $\PT(A)$ acts trivially on $A\wedge A$ and thus $H_3(\PB(A),\z)\simeq A\wedge A\simeq \z$. 
Now by an easy analysis of the main  spectral sequence we obtain the exact sequence
\[
\z \arr H_3(\PSL_2(\OO_d), \z)\arr \RB'(\OO_d) \arr 0.
\]
\end{exa}

\section{Homological description of the refined scissors congruence group}\label{sec4}

Let $\SM_2(A)$ denotes the group of monomial matrices in $\SL_2(A)$. Thus $\SM_2(A)$ 
consists of matrices 
$\begin{pmatrix}
a & 0\\
0 & a^{-1}
\end{pmatrix}$
and 
$\begin{pmatrix}
0 & a\\
-a^{-1} & 0
\end{pmatrix}$,
where $a\in \aa$. Let
\[
\PSM_2(A):=\SM_2(A)/\mu_2(A) I_2.
\]
For any $a\in\aa$, let
\[
\psi_1(a):=({\pmb \infty}, {\pmb 0}, {\pmb a})+({\pmb 0}, {\pmb \infty}, {\pmb a})
-({\pmb \infty}, {\pmb 0}, {\pmb 1})-({\pmb 0}, {\pmb \infty}, {\pmb 1}) \in \RP(A).
\]
Then 
\[
\lambda(\psi_1(a))=p_{-1}^+\Lan a\Ran,
\]
where $p_{-1}^+:=\lan-1\ran +1$ and $\Lan a\Ran:=\lan a\ran-1$ are elements of $\z[\GG_A]$. 
For more on $\psi_1(a)$ see \cite[\S3.2]{C-H2022}, \cite{hutchinson2013}, \cite{hutchinson2017} 
and \cite{hmm2022}.

\begin{thm}\label{general-exa}
Assume that $A$ satisfies the condition that $X_\bullet(A^2)\arr \z$ is exact in dimension $<2$.
If $\mu_2(A)=\{\pm 1\}$ and $H_i(\PT(A), \z)\simeq H_i(\PB(A),\z)$ for $i=2,3$, then we have the 
exact sequence
\[
H_3(\PSM_2(A),\mathbb{Z})\arr H_3(\PSL_2(A),\mathbb{Z}) \arr \RB(A)/\z[\GG_A]\psi_1(-1)\arr 0.
\]
\end{thm}
\begin{proof}
Let $\hat{X}_0(A^2)$ and $\hat{X}_1(A^2)$ be the free $\z$-modules generated by the sets
\[
\PSM_2(A)({\pmb \infty}):=\{(g{\pmb\infty}):g\in \PSM_2(A)\}
\]
and
\[
\PSM_2(A)({\pmb\infty}, {\pmb 0}):=\{(g{\pmb \infty}, g{\pmb 0}):g\in \PSM_2(A)\},
\]
respectively. It is easy to see that
\[
\hat{X}_1(A^2)  \overset{\hat{\partial}_1}{\arr} \hat{X}_0(A^2) \overset{\hat{\epsilon}}{\arr} \z \arr 0,
\]
is an exact sequence of $\PSM_2(A)$-modules, where $\hat{\partial}_1$ and $\hat{\epsilon}$ are defined 
similar to ${\partial}_1$ and $\epsilon$ discussed in Section \ref{sec1}, respectively. Moreover, 
$\hat{Z}_1(A^2):=\ker(\hat{\partial}_1)$ 
is equal to $\z\{({\pmb\infty},{\pmb 0})+({\pmb 0},{\pmb \infty})\}$. Observe that $\hat{Z}_1(A^2)\simeq\z$ 
and the action of $\PSM_2(A)$ on it is trivial. From the complex
\begin{equation}\label{comp2}
0 \arr\hat{Z}_1(A^2)\overset{\widehat{\inc}}{\arr} \hat{X}_1(A^2)\overset{\hat{\partial}_1}{\arr}
\hat{X}_0(A^2)\arr 0,
\end{equation}
we obtain the first quadrant spectral sequence
\[
\hat{E}^1_{p.q}=\left\{\begin{array}{ll}
H_q(\PSM_2(A),\hat{X}_p(A^2)) & p=0,1\\
H_q(\PSM_2(A),\hat{Z}_1(A^2)) & p=2\\
0, & p>2
\end{array}
\right.
\Longrightarrow H_{p+q}(\PSM_2(A),\z).
\]
The  complex (\ref{comp2}) is a $\PSM_2(A)$-subcomplex of (\ref{comp1}). This inclusion 
induces a natural morphism of spectral sequences
\begin{equation}\label{seqspmorph}
\begin{tikzcd}
\hat{E}^1_{p,q} \ar[d, "\inc_\ast"] \ar[r,Rightarrow] & H_{p+q}(\PSM_2(A),\z)\ar[d, "\inc_\ast"]\\
E^1_{p,q} \ar[r,Rightarrow] & H_{p+q}(\PSL_2(A),\z).
\end{tikzcd}
\end{equation}
As in case of $\PSL_2(A)$, we have
\[
\hat{X}_0(A^2)\simeq \Ind _{\PT(A)}^{\PSM_2(A)}\z, \ \ \ \ \ 
\hat{X}_1(A^2)\simeq \Ind _{\PT(A)}^{\PSM_2(A)}\z.
\]
Thus  by Shapiro's lemma we have
\[
\hat{E}_{0,q}^1 \simeq H_q(\PT(A),\z), \ \ \ \ \
\hat{E}_{1,q}^1 \simeq H_q(\PT(A),\z).
\]
Moreover, $\hat{d}_{1, q}^1=H_q(\hat{\sigma}) - H_q(\widehat{\inc})$, where $\hat{\sigma}(D(a))=wD(a)w^{-1}=D(a)^{-1}$. Thus
$\hat{d}_{1,0}^1$ is trivial, $\hat{d}_{1,1}^1$ is induced by the map $D(a)\mt D(a)^{-2}$ and 
$\hat{d}_{1,2}^1$ is trivial (see the paragraph above Lemma \ref{d21}). As in Lemma \ref{d21} we can show that $\hat{d}_{2,1}^1$ 
is trivial. Moreover, $\hat{d}_{1,3}^1$ can be studied similar to $d_{1,3}^1$ (see the paragraph after Lemma \ref{H3A}).

A direct calculation shows that the map $\hat{d}_{2,q}^1: H_q(\PSM_2(A),\z) \arr H_q(\PT(A),\z)$ 
is the transfer map \cite[\S 9, Chap. III]{brown1994}. Hence the composite
\[
H_q(\PSM_2(A),\z)\overset{\hat{d}_{2,q}^1}{\arr} H_q(\PT(A),\z)\overset{\inc_\ast}{\arr}H_q(\PSM_2(A),\z)
\]
coincides with multiplication by 2 \cite[Proposition 9.5, Chap. III]{brown1994}. In particular, 
$\hat{d}_{2,0}^1:\z \arr \z $ is multiplication by 2. This shows that 
\[
H_1(\PSM_2(A),\z)\simeq \z/2 \oplus \widetilde{\GG}_A,
\]
where $\widetilde{\GG}_A=\GG_A/\mu_2(A)$. If fact, from the extension
$1 \arr \PT(A) \arr \PSM_2(A) \arr \lan \overline{w}\ran \arr 1$, we obtain the split exact sequence 
\[
1 \arr \widetilde{\GG}_A \arr H_1(\PSM_2(A),\z) \arr \lan \overline{w}\ran \arr 1.
\]
Observe that $w^2=
\overline{\begin{pmatrix}
-1&0\\
0&-1
\end{pmatrix}}=1 \in \PT(A)$. 
Again a direct calculation shows that 
\[
\hat{d}_{2,1}^2: H_1(\PSM_2(A),\z)\arr H_2(\PT(A),\z)\simeq \PT(A) \wedge \PT(A)
\]
is trivial. In fact,
$\hat{d}_{2,1}^2|_{\widetilde{\GG}_A}$ is given by $\overline{\lan a\ran} \mapsto D(a)\wedge (\overline{-I_2})=0$ and  
$\hat{d}_{2,1}^2(\overline{w})=0$. Therefore from the spectral sequence $\hat{E}_{p,q}^ 1$
we obtain the isomorphism 
\[
H_2(\PSM_2(A),\z) \simeq \PT(A) \wedge \PT(A)\simeq \displaystyle\frac{\aa \wedge \aa}{\aa \wedge \mu_2(A)}.
\]
It follows from this that the image of the map
$\hat{d}_{2,2}^1:H_2(\PSM_2(A),\z)  \arr \PT(A)\wedge \PT(A)$
is $2 (\PT(A)\wedge \PT(A))$. Thus
\[
\hat{E}_{1,2}^2 \simeq \frac{\PT(A) \wedge \PT(A)}{2(\PT(A)\wedge \PT(A))}\simeq
\frac{\GG_A \wedge \GG_A}{\mu_2(A)\wedge \GG_A} , \ \ \ \hat{E}_{2,2}^2 
\simeq{}_2(\PT(A) \wedge \PT(A)).
\]
A direct calculation shows that the image of the map 
\[
\tilde{\GG}_A\oplus\langle\overline{\omega}\rangle\simeq H_1(\PSM_2(A),\z)\to 
H_1(\PSL_2(A),Z_1(A^2))\simeq \RP_1(A)
\]
is $\langle\psi_1(a^2),\psi_1(-1)\mid a\in A^{\times}\rangle$.

The morphism (\ref{seqspmorph}) induces morphism of filtrations
\[
\begin{array}{cccccccc}
0\se & \hat{F}_0 & \se &\hat{F}_1   &\se &\hat{F}_2   &\se & \hat{F}_3=H_3(\PSM_2(A),\z)\\
&\downarrow &      &\downarrow &    & \downarrow &    & \downarrow\\
0\se &    F_0    &\se   &  F_1      &\se & F_2        &\se & F_3 =H_3(\PSL_2(A),\z),
\end{array}
\]
where $E_{p,3-p}^\infty=F_p/F_{p-1}$  and $\hat{E}_{p,3-p}^\infty=\hat{F}_p/\hat{F}_{p-1}$.
Clearly $F_2=F_3$ and $\hat{F}_2=\hat{F}_3$. Consider the following commutative 
diagram with exact rows 
\begin{equation}\label{diagram}
\begin{tikzcd}
0 \ar[r] & \hat{F}_1 \ar[r] \ar[d] & H_3(\PSM_2(A),\z) \ar[r] \ar[d,"\inc_\ast "] 
& \hat{E}^{\infty}_{2,1} \ar[r] \ar[d] & 0 \\
0 \ar[r] & F_1 \ar[r] & H_3(\PSL_2(A),\z) \ar[r] & E^{\infty}_{2,1} \ar[r] & 0.
\end{tikzcd}
\end{equation}
We have seen that $\hat{E}^{\infty}_{2,1}\simeq  H_1(\PSM_2,\z)$ and $E^{\infty}_{2,1}\simeq\RB(A)$.
Now consider the commutative diagram with exact rows
\[
\begin{tikzcd}
0 \ar[r] & \hat{F}_0 \ar[r] \ar[d] & \hat{F}_1 \ar[r] \ar[d] & \hat{E}^{\infty}_{1,2} \ar[r] \ar[d] & 0 \\
0 \ar[r] & F_0 \ar[r] & F_1 \ar[r] & E^{\infty}_{1,2} \ar[r] & 0.
\end{tikzcd}
\]
Since $\hat{E}^1_{0,3}\simeq E^1_{0,3} $, the natural map  $\hat{F}_0 \arr F_0$ is surjective. Moreover, 
since $\hat{E}^1_{1,2}\simeq E^1_{1,2} $, the map $\hat{E}^{\infty}_{1,2}\arr E^{\infty}_{1,2}$ is 
surjective. These imply that the map $\hat{F}_1\arr F_1$ is surjective. Now  from the diagram (\ref{diagram})
we obtain the exact sequence
\[
H_3(\PSM_2(A),\mathbb{Z})\arr H_3(\PSL_2(A),\mathbb{Z}) \arr 
\frac{\mathcal{RB}(A)}{\langle\psi_1(a^2),\psi_1(-1)\mid a\in A^{\times}\rangle}\arr 0.
\]
Note that $\psi_1(a^2)=\Lan a\Ran \psi_1(-1)=\lan a\ran \psi_1(-1)-\psi_1(-1)$ (see \cite[Lemma 3.21]{C-H2022}). Thus
\[
\Big\langle\psi_1(a^2),\psi_1(-1)\mid a\in A^{\times}\Big\rangle
=\Big\lan \lan a\ran \psi_1(-1)\mid a \in \aa\Big\ran=\z[\GG_A]\psi_1(-1).
\]
This completes the proof of the theorem.

\end{proof}

\begin{thm}\label{general-exa-1}
Let $A$ satisfies the conditions of Theorem $\ref{general-exa}$. Then we have the isomorphism
\[
H_3(\PSL_2(A),\PSM_2(A);\z)\simeq \RP_1(A)/\z[\GG_A]\psi_1(-1).
\]
\end{thm}
\begin{proof}
From \cite[Proposition 7.1]{B-E--2023} and the morphism of complexes
\[
\begin{tikzcd}
0\ar[r]&\hat{Z}_1(A^2) \ar[r]\ar[d]&\hat{X}_1(A^2)\ar[r]\ar[d]&\hat{X}_0(A^2)\ar[r]\ar[d] &0\\
0 \ar[r] & Z_1(A^2)  \ar[r]  & X_1(A^2)  \ar[r]& X_0(A^2) \ar[r] & 0,
\end{tikzcd}
\]
we obtain the first quadrant spectral sequence
\[
\Eb_{p,q}^ 1\!=\!\!
\begin{cases}
H_q(\PSL_2(A), \PSM_2(A); X_p(A^2), \hat{X}_p(A^2)) & \text{$p=0,1$}\\
H_q(\PSL_2(A), \PSM_2(A); Z_1(A^2), \hat{Z}_1(A^2))  & \text{$p=2$}\\
0 & \text{$p>2$}
\end{cases}
\!\!\!\!\Rightarrow\! H_{p+q}(\PSL_2(A), \PSM_2(A); \z).
\]
We refer the readers to \cite[\S7]{B-E--2023} for the notation involved in the above spectral sequence.
Consider the long exact sequence
\[
\cdots\! \arr\! H_q(\PSM_2(A), \hat{X}_p(A^2))\! \arr\! H_q(\PSL_2(A),X_p(A^2))
\!\arr\! \Eb_{p,q}^1 \!\arr\! H_{q-1}(\PSM_2(A), \hat{X}_p(A^2)) 
\]
\[
\arr H_{q-1}(\PSL_2(A),X_p(A^2)) \arr \cdots.
\]
Since
\[
H_q(\PSL_2(A), X_0(A^2))\simeq H_q(\PB(A),\z), \ \ \ H_q(\PSL_2(A), X_1(A^2))\simeq H_q(\PT(A),\z),
\] 
\[
H_q(\PSM_2(A), \hat{X}_0(A^2))\simeq H_q(\PT(A),\z), \ \ \ H_q(\PSM_2(A), \hat{X}_1(A^2))\simeq H_q(\PT(A),\z),
\]
from the above exact sequence, for any $q$, we get 
\[
\Eb_{0,q}^1\simeq H_q(\PB(A),\PT(A);\z), \ \ \ \ \Eb_{1,q}^1=0.
\]
Therefore
\[
\Eb_{0,q}^2\simeq H_q(\PB(A),\PT(A);\z), \ \ \ \ \Eb_{1,q}^2=0, \ \ \ \ \Eb_{2,q}^2\simeq\Eb_{2,q}^1.
\]
Now by an easy analysis of the above spectral sequence we get the exact sequence
\[
\hspace{-3.5cm}
\cdots \!\arr\! H_{n+2}(\PSL_2(A),\! \PSM_2(A); \z) \!\arr\! \Eb_{2,n}^2\!\! \arr\! H_{n+1}(\PB(A),\!\PT(A);\z)\! 
\]
\[
\hspace{1.5cm}
\arr \!H_{n+1}(\PSL_2(A),\! \PSM_2(A); \z)\arr \Eb_{2,n-1}^2 \arr H_{n}(\PB(A),\PT(A);\z)\arr \cdots.
\]
Since $H_{2}(\PB(A),\PT(A);\z)=0=H_{3}(\PB(A),\PT(A);\z)$, we have
\[
\Eb_{2,1}^2\simeq  H_{3}(\PSL_2(A), \PSM_2(A); \z).
\]
On the other hand, we have the exact sequence
\[
H_1(\PSM_2(A), \z) \arr H_1(\PSL_2(A), Z_1(A^2)) \arr  \Eb_{2,1}^2 \arr 0.
\]
From these results we obtain the isomorphism
\[
H_3(\PSL_2(A), \PSM_2(A); \z)\simeq 
\frac{\mathcal{RP}_1(A)}{\langle\psi_1(a^2),\psi_1(-1)\mid a\in A^{\times}\rangle}.
\]
In the proof of the previous theorem, we showed that 
\[
\Big\langle\psi_1(a^2),\psi_1(-1)\mid a\in A^{\times}\Big\rangle=\z[\GG_A]\psi_1(-1).
\]
This completes the proof of the theorem.
\end{proof}

\begin{cor}\label{cor-last}
Let $A$ be a local domain such that its residue field either is infinite or has $p^d$ elements, 
where $(p-1)d>6$. Then we have the exact sequence
\[
H_3(\PSM_2(A),\mathbb{Z})\arr H_3(\PSL_2(A),\mathbb{Z}) \arr \RB(A)/\z[\GG_A]\psi_1(-1)\arr 0.
\]
Moreover,
\[
H_3(\PSL_2(A),\PSM_2(A);\z)\simeq \RP_1(A)/\z[\GG_A]\psi_1(-1).
\]
\end{cor}


\end{document}